\documentclass[twoside,11pt]{article}

\usepackage[latin1]{inputenc}
\usepackage[T1]{fontenc}

\usepackage{lipsum} 

\usepackage[english]{babel}
\usepackage{lmodern}
\rmfamily
\DeclareFontShape{T1}{lmr}{b}{sc}{<->ssub*cmr/bx/sc}{}
\DeclareFontShape{T1}{lmr}{bx}{sc}{<->ssub*cmr/bx/sc}{}

\usepackage{amssymb,amsmath,amsthm,amscd}
\usepackage{mathrsfs}

\usepackage{subfig}

\usepackage{tabularx} 
\usepackage{calc}
\usepackage{graphicx} 
\usepackage[nottoc]{tocbibind}

\usepackage{amsmath,amssymb}
\usepackage[usenames]{xcolor}
\usepackage{bbm}
\usepackage{dsfont}
\usepackage{mathrsfs}

\usepackage{varioref}
\usepackage[unicode,bookmarks,colorlinks=true,citecolor=blue,linkcolor=blue]{hyperref}
\usepackage[all]{hypcap}
\usepackage{cleveref}
\usepackage{hyperref}

\newtheorem{prop}{Proposition}%[theorem]

\crefname{prop}{Proposition}{Propositions}
\crefname{lem}{Lemma}{Lemmas}
\crefname{equation}{equation}{equations}
\crefname{cor}{Corollary}{Corollaries}

\DeclareMathOperator{\Vect}{\mathbf{span}}
\DeclareMathOperator{\cspan}{\overline{\mathbf{span}}}
\DeclareMathOperator{\Kern}{\mathbf{ker}}
\DeclareMathOperator{\im}{\mathbf{Im}}

\DeclareMathOperator{\supp}{\mathrm{supp}}
\begin{document}

\title{Kernel Spectral Clustering}

\author{Ilaria Giulini\footnote{The results presented in this paper 
were obtained while the author was preparing her PhD under the 
supervision of Olivier Catoni at the D\'epartement de Math\'ematiques et Applications, \'Ecole Normale Sup\'erieure, Paris, with the financial support of the 
R\'egion \^Ile de France.}\\
 ilaria.giulini@me.com} 
%\author{\name Ilaria Giulini \email ilaria.giulini@me.com \\
%       \addr INRIA Saclay\\
%       1, rue Honor\'e d'Estienne d'Orves\\
%       91120 Palaiseau, France
%}

\maketitle

\begin{abstract}%   <- trailing '%' for backward compatibility of .sty file
We investigate the question of studying spectral clustering in a Hilbert space where the set of points to cluster
are drawn i.i.d. according to an unknown probability distribution  
whose support is a union of compact connected components. 
We modify the algorithm proposed by Ng, Jordan and Weiss in \cite{NJW} in order to propose 
a new algorithm that automatically estimates the number of clusters and we 
characterize the convergence of this new algorithm in terms of convergence of Gram operators.
We also give a hint of how this approach may lead to learn transformation-invariant 
representations in the context of image classification.
\end{abstract}

%\begin{keywords}

\section{Introduction}

Clustering is the task of grouping a set of objects into classes, called {\it clusters}, in such a way that objects in the same group are more similar to each other than to those in other groups. 
Spectral clustering algorithms are efficiently used 
as an alternative to classical clustering algorithms, such as $k$-means, 
in particular in the case of 
not linearly separable data sets.
To perform clustering, these methods use the spectrum of some data-dependent matrices: the affinity matrix \cite{DonathHoffman}, or the Laplacian matrix \cite{MFied}.
Many different versions of spectral clustering algorithms can be found in the literature (see \cite{vL})
and one of the most successful algorithm 
has been proposed by Ng, Jordan and Weiss in \cite{NJW}. 
Given a set of points to cluster into $c$ classes and denoting by 
$A$ the affinity matrix and by 
$D$ the diagonal matrix whose $i$-th entry is the sum of the elements on the $i$-th row of $A$, their algorithm uses  
the $c$ largest eigenvectors of the Laplacian matrix $D^{-1/2}AD^{-1/2}$ simultaneously. 
More precisely, the data set is first embedded in a $c$-dimensional space
in which clusters are more evident 
and then points are separated using $k$-means, or any other standard algorithm. \\
Other algorithms use different renormalizations of the affinity matrix. 
For example, Shi and Malik in \cite{ShiMalik} use the second smallest eigenvector of the unnormalized Laplacian matrix $D -A$
to separate the data set into two groups and use the algorithm recursively to get more than two classes, whereas 
Meila and Shi in \cite{MeS} use the first $c$ largest eigenvectors of the stochastic matrix $D^{-1}A$
(that has the same eigenvectors as the normalized Laplacian 
matrix $\mathrm I - D^{-1}A$) to compute a partition in $c$ classes. 
These algorithms treat the question as a graph partitioning problem and they are based on the so called normalized cut criterion.
Different cost functions can be found in the literature (see \cite{CSTK_cl}, \cite{BJ})
and more generally
the study of Laplacian matrices has been carry out also in different contexts, 
as for semi-supervised learning \cite{CSZ} and manifold learning \cite{BNy}.
\\[1mm]
We consider the setting of performing spectral clustering in a Hilbert space. 
This general framework includes the analysis of both data sets in a functional space
and samples embedded in a reproducing kernel Hilbert space. The latter is the case of 
kernel methods that use the kernel trick to embed the data set into a reproducing kernel Hilbert space 
in order to get a new representation that simplifies the geometry of the classes. 
Our point of departure is the algorithm proposed by Ng, Jordan and Weiss \cite{NJW} 
but interpreted in a infinite-dimensional setting, so that we view the matrices described above as 
empirical versions of some underlying integral operators. 
We assume that the points to cluster are drawn according to an unknown probability distribution 
whose support is a union of compact connected components (see \cite{vLBB} for consistency of clustering algorithms). 
Our idea is to view spectral clustering as a change of representation in a reproducing
kernel Hilbert space, induced by a change of kernel, and to propose a new algorithm 
that automatically estimates the number of clusters. Usually the number of clusters is assumed to be
know in advance (see \cite{MeS}, \cite{NJW}) 
or it is linked to the presence of a sufficient large gap in the spectrum of the Laplacian matrix. 
To achieve our aim we 
replace the projection on the space spanned by the 
largest eigenvectors of the Laplacian matrix proposed by Ng, Jordan and Weiss 
with a suitable power of the Laplacian operator.
Indeed, such an iteration performs some kind of soft truncation of the eigenvalues 
and hence leads to a natural dimensionality reduction.
This iteration is related to the computation of the 
marginal distribution at time $t$ of some Markov chains with exponential 
rare transitions, as suggested by \cite{Catoni01}
in the case where the unknown probability distribution has a finite support.  
We conjecture (and this will be the subject of a future work) that the same kind of behavior holds for more general supports 
and we hint that spectral clustering, coupled with some preliminary change of representation in a reproducing kernel Hilbert space, can be a winning tool to bring down the representation of classes to a low-dimensional space
and may lead to a generic and rather radical alternative. 
This suggests that the kernel trick, introduced to better separate classes in the supervised learning framework addressed by
support vector machines (SVMs), is also feasible in the unsupervised context, where we do not separate classes using hyperplanes in the feature 
space but instead we use spectral clustering 
to perform the classification.
We also suggest with an example how this approach may lead to learn transformation-invariant representations
of a same pattern. \\
Developing a convincing toolbox for unsupervised invariant-shape analysis is beyond the scope of this study and it will be carried on elsewhere. 
However we observe that the pattern transformations we would like to take into account in image analysis are numerous and not easy to formalize:
they may come from some set of transformations such as translations, rotations or scaling or they may come 
from the conditions in which the images have been taken, for example,
changes in the perspective or in illumination, partial occlusions, object deformations, etc. 
Making a representation invariant with respect to a set of transformations is a challenging task even in the simpler case of translations. 
Indeed, it is sufficient to observe that the set of functions obtained by translating a single pattern in various directions typically 
spans a vector space of high dimension, meaning that the shapes (here the functions) 
that we would like to put in the same category do not even live in a common low-dimensional subspace.
A possible approach is to study representations that leave invariant some group of transformations, for instance, the Fourier transform that has translation invariant properties, since its modulus is translation invariant. However it is unstable to small deformations at high frequencies. Wavelet transforms provide a workaround. Scattering representations proposed by Mallat \cite{Mallat} compute translation-invariant representations by cascading wavelet transforms and modulus pooling operators. They bring improvements for audio \cite{Mallat_Anden} and for image \cite{Mallat_Bruna} classification. 
This kind of careful mathematical study has to be repeated for any kind of transformations. 
Instead of deriving the representation of a pattern from a mathematical study, the idea here is to learn the representation itself from examples of patterns that sample in a sufficiently dense way the orbits of the set of transformations at stake.\\[1mm]
The paper is organized as follows. 
In Section~\ref{sec1} we introduce the ideal version of our algorithm that uses the underlying unknown probability distribution
and we provide an interpretation of the clustering effect in terms of Markov chains.  
In Section~\ref{sec2} we introduce an empirical version of the algorithm and we provide some convergence results, based 
on the convergence of some Gram operators.
Finally experiments are shown in Section~\ref{sec3}.

\section{The Ideal Algorithm}\label{sec1}
Let $\mathcal X$ be a compact subset of some separable Hilbert space endowed 
with the (unknown) probability distribution $\mathrm P$ on $\mathcal  X$ 
and assume that the support of $\mathrm{P}$ is made of several compact connected components.
Let $A : \mathcal  X \times \mathcal  X \to \mathbb R$ be a symmetric positive semi-definite 
kernel, normalized in such a way that $A(x,x) = 1$, for any $x$ in the state
space $\mathcal{  X}$, and let us define the symmetric positive semi-definite kernel
\[
K(x,y)= A(x,y)^2. 
\] 
This kernel plays the same role of the affinity matrix, while the Laplacian matrix is 
now replaced by the integral operator with kernel
\[
\overline K (x,y) = \mu(x)^{-1/2} \ K(x,y) \ \mu(y)^{-1/2}
\]
where the function 
\begin{equation}
\label{eq4.6}
\mu(x) = \int K(x,z) \mathrm d \mathrm P(z)
\end{equation}
is the analogous of the diagonal matrix $D$. According to the Ng, Jordan and Weiss algorithm 
we consider more specifically the case where
\[
K(x,y) = K_{\beta}(x,y) = \exp \bigl( - \beta \lVert x - y \rVert^2 \bigr), \quad \beta>0.
\]

\subsection{The Algorithm}
The ideal algorithm (that uses the unknown probability distribution $\mathrm P$) goes as follows. For any $x, y \in \mathcal  X,$

\begin{enumerate}
\item Form the kernel 
\[
\overline K(x,y) =  \mu(x)^{-1/2} K(x,y)\mu(y)^{-1/2}
\]
where $\mu(x) = \int K(x,z) \mathrm d \mathrm P(z)$ 

\item Construct the new (iterated) kernel 
\[
\overline K_m(x,y) = \int \overline K(y,z_1) \overline K(z_1, z_2) \dots \overline K(z_{m-1},x) \  \mathrm d \mathrm P^{\otimes (m-1)}(z_1, \dots ,z_{m-1}),
\]
 where $m >0 $ is a free parameter
 
\item Make a last change of kernel (normalizing the kernel $\overline K_m$) 
\[
K_m(x,y) =\overline K_{2m}(x,x)^{-1/2} \ \overline K_{2m}(x,y) \ \overline K_{2m}(y,y)^{-1/2}
\]

\item Cluster points according to the new representation defined 
by the kernel $K_m$. 
\end{enumerate}

\vskip1mm
\noindent
As it is suggested in the next section, the representation induced by the kernel $K_m$
makes the subsequent classification an easy task.

\subsection{The Clustering Effect}

The main idea is that, in the Hilbert space defined by the kernel $K_m$, clusters are concentrated around  
orthogonal unit vectors, forming the vertices of a regular simplex. 
To give an intuition of this clustering effect, we provide a Markov chain analysis of the algorithm, according to some ideas 
suggested by \cite{Catoni01}. 
\\[1mm]
Assume that the scale parameter $\beta$ in the Gaussian kernel $K$ is large enough. For the choice of $\beta$ we refer to Section~\ref{choice_beta}.\\
Define the kernel 
\[
M(x,y) = \mu(x)^{-1} K(x,y)
\]
where $\mu$ is defined in ~\cref{eq4.6} and consider the corresponding integral operator 
\[
\mathbf{M}(f): x \mapsto \int M(x,z) f(z) \, \mathrm{d}\mathrm{P}(z).
\]
Observe that $\mathbf{M}$ is the transition operator 
of a Markov chain $(Z_k)_{k \in \mathbb{N}}$ on $\mathcal{X}$ and 
that
\[
M(x,y) = \frac{\mathrm{d} 
\mathbb{P}_{Z_1 | Z_0 = x}}{
\mathrm{d} \mathrm{P}}(y).
\] 
So the Markov chain related to the operator $\mathbf{M}$ 
has a small probability to jump from one component to another one.
As already said, from ~\cite{Catoni01} it can be deduced that, when the support of $\mathrm P$
is finite, 
for suitable values of $m$, depending on $\beta$ as $\exp(\beta T^2)$, 
 the measure 
\[
\mathbf{M}_x^m: f \mapsto \mathbf{ M}^m(f)(x)
\]
is close (for large enough values of $\beta$) to be supported by 
a cycle of depth larger than $H$ of the state space. 
The cycle decomposition of the state space of a Markov 
chain with exponential transitions is some discrete counterpart 
of the decomposition into connected components. 
More precisely, the measure $\mathbf{M}_x^m$ is close to 
the invariant measure of the operator $\mathbf{M}$ restricted 
to the connected component which $x$ belongs to.  
As a result, the function $x \mapsto \mathbf{M}^m (f)(x) $ is approximately constant on each connected component.
Thus, for reasons that we will not try to prove here, but that 
are confirmed by experiments, we 
expect the same kind of behaviour in the general setting. 
To rewrite this conjecture in terms of the kernel $K_m$ we introduce the following proposition.

\begin{prop}\label{lem_3rep} 
Let  $\mathrm{Q}$ be the invariant measure of the Markov transition 
operator $\mathbf{M}$ and define 
\[
\mathfrak{R}_x = \mu(x)^{1/2} \frac{\mathrm{d}\mathbf{M}^m_x}{\mathrm{d}\mathrm{Q}} \in L^2_{\mathrm{Q}}, \qquad x\in \mathcal X.
\]
It holds that 
\[
K_m(x,y) = \Bigl\langle \frac{\mathfrak{R}_x }{\lVert \mathfrak{R}_x \rVert_{L^2_\mathrm{Q}}},  \frac{\mathfrak{R}_y }{\lVert \mathfrak{R}_y \rVert_{L^2_\mathrm{Q}}}  \Bigr\rangle_{L^2_\mathrm{Q}}.
\]
\end{prop} 

\begin{proof} 
Using the fact that $\overline K= \overline K_1$ and that 
\begin{equation}\label{kmkm}
\overline K_m(x,y) = \int \overline K(x,z) \overline K_{m-1}(z,y) \, \mathrm d \mathrm P(z),
\end{equation}
by induction we get
\[ 
\mathbf{M}^m(f)(x) = \mu(x)^{-1/2} \int \overline K_m(x,z) \mu(z)^{1/2} f(z) \, \mathrm{d}\mathrm{P}(z).
\] 
Thus the measure $\mathbf{M}^m_x$ has density
\[ 
\frac{\mathrm{d}\mathbf{M}^m_x}{\mathrm{d}\mathrm{P}} (y) = \mu(x)^{-1/2} \overline K_m(x,y)\ \mu(y)^{1/2}.
\] 
Moreover, since 
\[
\int \mu(x) M(x,y) f(y) \, \mathrm{d}\mathrm{P}(x) \mathrm{d}\mathrm{P}(y) = 
\int K(x,y) f(y) \, \mathrm{d}\mathrm{P}(x) \mathrm{d}\mathrm{P}(y) = \int 
\mu(y) f(y) \, \mathrm{d}\mathrm{P}(y),
\]
the invariant measure $\mathrm{Q}$ has a density 
with respect to $\mathrm{P}$ equal to $\frac{\mathrm{d}\mathrm{Q}}{\mathrm{d}\mathrm{P}} = \mu.$
As a consequence
\[ 
\frac{\mathrm{d}\mathbf{M}^m_x}{\mathrm{d}\mathrm{Q}}(y)= \mu(x)^{-1/2} \overline K_m(x,y)\ \mu(y)^{-1/2}.
\] 
According to ~\cref{kmkm} and recalling the definition of $K_m$ we conclude the proof.
$\blacksquare$
\end{proof}

\vskip2mm
\noindent
With these definitions, 
our conjecture can be formulated as 
\[
\lim_{{\beta } \to \infty} K_{\exp({\beta} T^2)} (x,y)  = \sum_{C \in \mathcal C_T} \mathbbm 1\left(\{ x,y\} \subset C\right),
\]
where $\mathcal C_T$ denotes the connected components of the graph $\left\{ x,y \in \mathrm{supp}(\mathrm P) \ | \ \| y-x\| <T \right\}.$ 
Remark that, as we assumed that the support of $\mathrm{P}$ is a union of 
compact topological connected components, taking $T$ to be less than 
the minimum distance between two topological components of $\supp (\mathrm{P})$, 
we obtain that $\mathcal{C}_T$ coincides with the topological components of 
$\supp (\mathrm{P})$. 
As a consequence, for suitable values of the scale parameter $\beta$ and of the number of iterations $m$,
the kernel $K_m(x,y)$ is close to zero (or equivalently the supports of the probability measures $\mathbf{M}^m_x$ and $\mathbf{M}^m_y$ are almost disjoint) 
if $x$ and $y$ belong to two different clusters, 
whereas it is close to one (or equivalently the supports of $\mathbf{M}^m_x$ and $\mathbf{M}^m_y$ are almost the same) when $x$ and $y$ belong to the same cluster. 
Moreover since the kernel $K_m(x,y)$ is the cosine of the angle formed by the two vectors representing $x$ and $y$,
according to the conjecture, it is either close to zero or close to one,
showing that, in the Hilbert space defined by $K_m$, clusters are concentrated around  
orthogonal unit vectors.

\subsection{Choice of the Scale Parameter}\label{choice_beta}

We conclude this section with some remarks on the choice of the scale parameter $\beta.$
In the case where the influence kernel $K$ is of the form 
\[ 
K(x, y) = K_{\beta}(x, y) = \exp \bigl( - \beta \lVert x - y \rVert^2 \bigr),
\] 
we propose to choose $\beta$ as the solution of the equation 
\[
F(\beta) := \int K_{\beta}(x, y)^2 \, \mathrm{d}\mathrm{P}(x)  \mathrm{d}\mathrm{P}(y) =  h
\] 
where $h$ is a suitable parameter which measures 
the probability that two independent points drawn according to the probability $\mathrm{P}$ 
are close to each other.
Introducing the Gram operator $L_\beta : L^2_{\mathrm{P}} \to L^2_{\mathrm{P}}$ defined by the kernel $K_{\beta}$ as
\[
L_{\beta}(f)(x)  = \int K_{\beta}(x, y) f(y) \, \mathrm{d}\mathrm{P}(y), \quad  x \in \mathcal{  X},
\]
the parameter $h$ is equal to the square of the Hilbert-Schmidt norm of $L_{\beta}.$
Observe that $L_\beta$ has a discrete spectrum $\lambda_1 \geq 
\lambda_2 \geq \dots \geq 0$ that satisfies, according to the Mercer theorem, 
\begin{align*}
\sum_{i=1}^{+ \infty} \lambda_i & = \int K_\beta(x, x) \, \mathrm{d}\mathrm{P}(x) = 1,\\ 
\sum_{i=1}^{+ \infty} \lambda_i^2  & = F(\beta) \leq 1
\end{align*}
since $K_{\beta}(x,x) = 1$. We also observe that
\[
\lim_{\beta \rightarrow 0} F(\beta) = 1
\]
which implies that
for $\beta$ going to zero, $\lambda_1$ goes to one whereas all the other eigenvalues $\lambda_i$, for $i \geq 2$, go to zero.
Moreover, 
\[
\lim_{\beta \rightarrow + \infty} F(\beta) = 0, 
\]
so that when $\beta$ grows, the eigenvalues are spread more and more 
widely. 
Therefore, $F(\beta)$ governs the spread of the eigenvalues of $L_{\beta}$. 
For these reasons, the value of the parameter $h = F(\beta)$ 
controls the effective dimension of the representation of the distribution of points $\mathrm{P}$ 
in the reproducing kernel Hilbert space defined by 
$K_{\beta}$. Experiments show that this effective dimension has 
to be pretty large in order to have a proper clustering effect, meaning that we will impose a small value of the 
parameter $h$.  \\[1mm]
Note that in the experiments we do not have access to the function $F(\beta)$ but we have to consider 
its empirical version
\[
\frac{1}{n(n-1)} \sum_{1\leq i<j\leq n} K_\beta(X_i, X_j)^2
\]
where $X_1, \dots,  X_n$ is an i.i.d. sample drawn from the 
probability distribution $\mathrm{P}$.

\section{The Empirical Algorithm}\label{sec2}

In order to provide an empirical version of the ideal algorithm proposed in Section~\ref{sec1},
we connect the previous kernels with some Gram operators. 
Note that by {\it empirical algorithm} we mean an algorithm 
based on the sample distribution instead of on the unknown probability distribution $\mathrm{P}$.\\[1mm]
From now on, given a Hilbert space $\mathcal K$ and a probability distribution $\mathcal P$ 
on $\mathcal K$,
the Gram operator $\mathfrak G: \mathcal K\to \mathcal K$ is defined as 
\[
\mathfrak G u = \mathbb E_{Y \sim \mathcal P} \big[\langle u, Y\rangle_{\mathcal K} \ Y  \big].
\]

\vskip1mm
\noindent
The next proposition links the function $\mu$ defined in ~\cref{eq4.6} to the Gram operator defined in the reproducing kernel
Hilbert space induced by the kernel
\[
A(x,y) = K(x,y)^{1/2} = \exp\left( -\beta\| x-y\|^2/2  \right).
\]

\begin{prop}\label{gA}
Let $\mathcal H_A$ be the reproducing kernel Hilbert space defined by $A$ and $\phi_A: \mathcal X \to \mathcal H_A$ the corresponding feature map.
Denote by $\mathcal G_A: \mathcal H_A \to \mathcal H_A$ the Gram operator
\[
\mathcal G_A v = \int \langle v, \phi_A(z)\rangle_{\mathcal H_A} \ \phi_A(z) \, \mathrm d \mathrm P(z).
\]
Then 
\[
\mu(x) = \int \langle  \phi_A(x), \phi_A(z)\rangle_{\mathcal H_A}^2 \, \mathrm d \mathrm P(z) = \langle \mathcal G_A \phi_A(x), \phi_A(x) \rangle_{\mathcal H_A}.
\]
\end{prop}

\begin{proof}
The result follows recalling that by the Moore-Aronszajn theorem
\[
A(x,y) = \langle \phi_A(x), \phi_A(y) \rangle_{\mathcal H_A}.
\]
$\blacksquare$
\end{proof}

\noindent
A similar result relates the iterated kernel $\overline K_m$ to the Gram operator
defined in the reproducing kernel Hilbert space induced by the Gaussian kernel $K$.

\begin{prop}\label{prop_Km}
Let $\mathcal H$ be the reproducing kernel Hilbert space defined by $K$ and $\phi_K: \mathcal X \to \mathcal H$ the corresponding feature map.
Define
\begin{equation}
\label{phibar}
\phi_{\overline K}(x) = \mu(x)^{-1/2} \phi_{K}(x)
\end{equation}
and consider the Gram operator $\mathcal G: \mathcal H \to \mathcal H$ 
\begin{equation}\label{defGop}
\mathcal G v =  \int \langle v, \phi_{\overline K}(z) \rangle_{\mathcal H}\  \phi_{\overline K}(z) \ \mathrm d \mathrm P(z).
\end{equation}
Then, for any $m>0,$ 
\[
\overline K_{2m}(x,y)= \langle \mathcal{  G}^{m-1/2} \phi_{\overline K}(x),\mathcal{  G}^{m-1/2} \phi_{\overline K}(y)\rangle_{\mathcal H}.
\]
\end{prop}

\begin{proof}
According to the Moore-Aronszajn theorem
\[
K(x,y) = \langle \phi_K(y), \phi_K(x) \rangle_{\mathcal H}
\]
and thus, by definition, the kernel $\overline K(x,y) =  \mu(x)^{-1/2} \ K(x,y) \ \mu(y)^{-1/2}$ can be written as
\[
\overline K(x,y) = \langle \mu(x)^{-1/2} \phi_{K}(x), \mu(y)^{-1/2} \phi_{K}(y) \rangle_{\mathcal H} = \langle \phi_{\overline K}(y), \phi_{\overline K}(x) \rangle_{\mathcal H}.
\]
Using the fact that $\overline K= \overline K_1$ and that
\[
\overline K_m(x,y) = \int \overline K(x,z) \overline K_{m-1}(z,y) \, \mathrm d \mathrm P(z),
\]
by induction we get that 
\[
\overline K_m(x,y)= \langle \mathcal G^{m-1} \phi_{\overline K}(x), \phi_{\overline K}(y)\rangle_{\mathcal H}.
\]
Since $\mathcal G$ is a positive symmetric operator we conclude the proof. 
$\blacksquare$
\end{proof}

\vskip2mm
\noindent
As a consequence, 
the representation of $x \in \mathcal{  X}$ defined by the renormalized 
kernel $K_m$ is isometric to the representation 
$\lVert \mathcal{  R}_x \rVert^{-1}_{\mathcal{  H}}
\mathcal{  R}_x \in \mathcal{  H}$, where
\[
\mathcal{  R}_x = \mathcal{  R}_x(m) = \mathcal{  G}^{m-1/2} \phi_{\overline K}(x) \in \im(\mathcal{  G}) \subset \mathcal{  H}. 
\]

\vskip2mm
\noindent
{\bf Remark 4.} {\it
The representation of $x$ in the kernel space 
defined by $K_m$ is also isometric to the representation 
$\lVert \mathbf{R}_x \rVert^{-1} \mathbf{R}_x \in L_{\mathrm{P}}^2$, 
where $\mathbf{R}_x \in L_{\mathrm{P}}^2$ is defined as 
\begin{equation}
\label{eq:rep}
\mathbf{R}_x(z) = \overline{K}_m(x,z) = 
\langle \mathcal{  G}^{m-1} \phi_{\overline K}(x), \phi_{\overline K}(z) \rangle_{\mathcal{  H}}.
\end{equation}
This is a consequence of the fact that 
\begin{align*}
\langle \mathbf{R}_x, \mathbf{R}_y \rangle_{L^2_{\mathrm P}} = \int \overline{K}_m(x,z) \overline{K}_m(z,y) \, \mathrm{d}\mathrm{P}(z) = \overline{K}_{2m}(x,y).
\end{align*}
}

\subsection{An Intermediate Step}\label{emp_algo}

As already said, by the Moore-Aronszajn theorem, the Gaussian kernel $K$ defines a reproducing kernel Hilbert space $\mathcal H$
and a feature map $\phi_K: \mathcal X \to \mathcal H$
such that
\[
K(x,y) = \langle \phi_K(y), \phi_K(x) \rangle_{\mathcal H}.
\]
We introduce an intermediate version of the algorithm, that uses the 
feature map $\phi_K$. Since this feature map is not explicit, we 
will afterward translate this description into an algorithm 
that manipulates only scalar products and not implicit feature 
maps. 
This intermediate step is useful to provide the convergence results presented in Section~\ref{cr}.\\[1mm]
Let $X_1, \dots, X_n$ be the set of points to cluster, that we assume to be drawn i.i.d. according to $\mathrm P$. 
The algorithm goes as follows.

\begin{enumerate}
\item Form the kernel 
\[
\widehat K(x,y) = \langle \phi_{\widehat K}(x) , \phi_{\widehat K}(y) \rangle_{\mathcal H} 
\]
where 
\[
\phi_{\widehat K}(x) =  \widehat \mu(x)^{-1/2}\phi_{K}(x)\qquad \text{and} \qquad
\widehat \mu(x)  = \frac{1}{n} \sum_{i=1}^n K(x,X_i)
\]

\item Construct the kernel 
\begin {equation}
\label{hatK_m}
\widehat K_m(x,y) = \langle \widehat{\mathcal Q}^{\frac{m-1}{2}} \phi_{\widehat K}(x),\widehat{\mathcal Q}^{\frac{m-1}{2}} \phi_{\widehat K}(y)\rangle_{\mathcal H}
\end{equation}
where
\[
\widehat{\mathcal Q} v = \frac{1}{n} \sum_{i=1}^n \widehat  \mu(X_i)^{-1} \langle v, \phi_{K}(X_i) \rangle_{\mathcal H}\  \phi_{K}(X_i).
\]
The definition of $\widehat{\mathcal Q} $ is justified in Remark 5.
 
\item Make a last change of kernel and consider
\[
H_m(x,y) = \widehat K_{2m}(x,x)^{-1/2} \ \widehat K_{2m}(x,y)\ \widehat K_{2m}(y,y)^{-1/2}
\]

\item Cluster points according to this new representation, by thresholding the distance between points.\\[1mm]

\end{enumerate}

\noindent
{\bf Remark 5.} 
{\it 
The construction of the estimator $\widehat {\mathcal Q}$ follows from a two-steps estimate
of the Gram operator $\mathcal G$ defined in ~\cref{defGop}. Indeed, according to the definition of the feature map $\phi_{\overline K}$, 
the Gram operator $\mathcal G$ rewrites as
\[
\mathcal G v = \int \mu(z)^{-1} \langle v, \phi_{K}(z) \rangle_{\mathcal H}\  \phi_{K}(z) \ \mathrm d \mathrm P(z).
\]
Thus first we replace the function $\mu$ with its estimator $\widehat \mu$ and then we replace 
the unknown distribution $\mathrm P$ with the sample distribution.
}

\subsection{Implementation of the Algorithm}\label{impa}

We now describe the algorithm used in the experiments. 
As in the previous section, let $X_1, \dots, X_n$ be the set of points to cluster, 
drawn i.i.d. according to $\mathrm P$. 

\begin{enumerate}
\item Construct for $i,j =1,\dots,n$
\[
K_{ij} = \frac{1}{n} \exp \bigl( - \beta \lVert X_i - X_j \rVert^2 \bigr),
\]
where the parameter $\beta$ is chosen, according to Section~\ref{choice_beta}, as the solution of
\[
\frac{1}{n(n-1)} \sum_{\substack{i,j\\i\neq j}} \exp \bigl( - 2 \beta \lVert X_i - X_j \rVert^2 \bigr) \simeq 0.005
\] 

\item Define the diagonal matrix $D = \mathrm{diag}(D_1, \dots, D_n)$ where
\[
D_i = \max \left\{ \frac{1}{n} \sum_{j=1}^n K_{ij}, \ \sigma\right\} \quad {\text{with}} \ \sigma= 0.001
\]

\item Form the matrix 
\[
M = D^{-1/2} K D^{-1/2} 
\]

\item Consider
\begin{equation}\label{eqC}
C_{ij} = (M^m)_{ii}^{-1/2}  (M^m)_{ij}(M^m)_{jj}^{-1/2}, \qquad i,j =1,\dots,n
\end{equation}

\vskip1mm
The choice of the number of iterations $m$ is done automatically
and it is described in Remark 6.

\item Cluster points according to the representation induced by $C$, as detailed in Remark 7.
\end{enumerate}

\vskip5mm
\noindent
Note that in some configurations the above algorithm obviously yields an unstable classification, 
but in practice, when the classes are clearly separated from each other, 
this simple scheme is successful.

\vskip5mm
\noindent
{\bf Remark 6.} {\it
To choose automatically the number of iterations, we have to fix a maximal number of clusters. 
This means that we have to provide an upper bound, that we denote by $p$, on the number of classes we expect to have. However, 
as it can be seen in the simulations, this choice of $p$ is robust, meaning that $p$ can be harmlessly overestimated.
Thus, denoting by $\widehat{\lambda}_1 \geq \cdots \geq \widehat{\lambda}_n \geq 0$
the eigenvalues of $M$, 
we choose the number $m$ of iterations by solving 
\[
\biggl( \frac{\widehat{\lambda}_p}{\widehat{\lambda}_1} \biggr)^m \simeq \zeta
\]
where $\zeta >0$ is a given small parameter.
The choice of $\zeta$ is also robust and $1/100$ is a reasonable value for it. 
}

\vskip5mm
\noindent
{\bf Remark 7.} {\it 
Our (greedy) classification algorithm consists in taking an index $i_1 \in\{1, \dots, n\}$ at random, in forming the corresponding class 
\[
\widehat{C}_{i_1} = \Bigl\{ \, j \ | \ {C}_{i_1 j} \geq s \, \Bigr\},
\]
where $s$ is a threshold that we take equal to $0.1$ in all our experiments,
and then in starting again with the remaining indices.
In such a way we construct a sequence of indices $i_1, \dots, i_c,$ where, as a consequence, the number of clusters $c$ is automatically estimated. 
The $k$-th class is hence defined as
\[
\widetilde{C}_ k = \widehat{C}_{i_k} \setminus \bigcup_{\ell < k} \widehat{C}_{i_{\ell}}.
\]
}

\subsection{Convergence Results}\label{cr}

We introduce some results on the accuracy of the estimation of the ideal algorithm through the empirical one.
In particular we compare the ideal iterated kernel $\overline{K}_{2m}$
with its estimator $ \widehat{  K}_{2m}$ defined in ~\cref{hatK_m}.\\

\noindent
{\bf Proposition 8.} {\it
Define
\begin{equation}\label{def_chisc}
\chi(x) = \left( \frac{\mu(x)}{\widehat \mu (x)}\right)^{1/2}.
\end{equation}
For any $x, y \in \supp(\mathrm{P}),$ $m>0$,
\begin{multline*}
 \bigl\lvert \widehat{  K}_{2m}(x,y) - \overline{K}_{2m}(x,y)\bigr\rvert  \leq \frac{\displaystyle \max \{1, \lVert \chi \rVert_{\infty}\}^2}{\mu(x)^{1/2} \mu(y)^{1/2}} \Biggl( (2m-1) \lVert \widehat{  \mathcal{  Q}} - \mathcal{  G} \rVert_{\infty}  \Bigl( 1 + \lVert 
\widehat{  \mathcal{  Q}} - \mathcal{  G} \rVert_{\infty} \Bigr)^{2m-2}\\
+ 2 \lVert \chi - 1 \rVert_{\infty} \Biggr) 
\end{multline*}
and
\begin{multline*}
\left\| \widehat{  K}_{2m}(x,\cdot) - \overline{K}_{2m}(x,\cdot) \right\|_{L^2_{\mathrm P}}   \leq \frac{\displaystyle \max \{ 1, \lVert \chi \rVert_{\infty}\}^2}{\mu(x)^{1/2}} \Biggl( (2m-1) \lVert \widehat{  \mathcal{  Q}} - \mathcal{  G} \rVert_{\infty}  \Bigl( 1 + \lVert 
\widehat{  \mathcal{  Q}} - \mathcal{  G} \rVert_{\infty} \Bigr)^{2m-2} \\+ 2 \lVert \chi - 1 \rVert_{\infty} \Biggr).
\end{multline*}
}

\vskip1mm
\noindent
For the proof we refer to Section~\ref{pf1}. 

\vskip2mm
\noindent
Other convergence results are related to the description of the kernel $K_m$ in terms of the representation $\mathbf{R}_x \in L^2_{\mathrm P}$ introduced in 
~\cref{eq:rep}. More precisely, we recall that according to ~\cref{prop_Km}
\[
K_m(x,y) = \Bigl\langle \lVert \mathbf{R}_x \rVert^{-1} \mathbf{R}_x, \lVert \mathbf{R}_y \rVert^{-1} \mathbf{R}_y \Bigr\rangle_{L^2_{\mathrm P}},
\]
where $\mathbf{R}_x(z) = \langle \mathcal{  G}^{m-1} \phi_{\overline K}(x), \phi_{\overline K}(z) \rangle_{\mathcal{  H}}$ and the feature map
$ \phi_{\overline K}$ is defined in ~\cref{phibar}.
Our estimated representation of 
$x$ in $L_{ \mathrm{P}}^2$   
is $ \widehat{  \mathbf{N}}(x)^{-1}  
\widehat{  \mathbf{R}}_x$, where  
\begin{align*}
\widehat{  \mathbf{R}}_x(y)  = \bigl\langle \widehat{  \mathcal{  Q}}^{m-1} \phi_{\widehat K}(x), \phi_{\overline K}(y) \bigr\rangle_{\mathcal{  H}} \quad  \text{ and } \quad
\widehat{  \mathbf{N}}(x) =  \langle \widehat{  \mathcal{  Q}}^{2m-1}  \phi_{\widehat K}(x), \phi_{\widehat K}(x) \rangle_{\mathcal{  H}}^{1/2}.
\end{align*}
This representation is not fully observable because of the presence of the ideal feature map $\phi_{\overline K}$ in the definition of $\widehat{  \mathbf{R}}_x$.
Nevertheless, it can be used in practice
since the representation $\widehat{  \mathbf{R}}_x \in L_{ \mathrm{P}}^2$
is isometric to the fully observed representation 
\[
\widehat{  \mathcal{  R}}_x  = \widehat{  \mathcal{  Q}}^{m-1} \phi_{\widehat K}(x), \qquad x \in \mathcal{  X},
\] 
in the Hilbert space $\bigl( \im (\mathcal{  G}), \lVert \cdot \rVert_{\mathcal{  G}} \bigr)$ 
with the non-observable Hilbert norm $\lVert u \rVert_{\mathcal{  G}} = \langle \mathcal{  G} u, u \rangle^{1/2}$, that could be estimated by $\langle \widehat{  \mathcal{  Q}} u, u \rangle_{\mathcal H}^{1/2}$. For further details we refer to Section~\ref{prof1}.\\[1mm]
In the following we provide non-asymptotic bounds 
for the error of approximating 
of the ideal non-normalized representation $\mathbf{R}_x$ with the estimated non-normalized representation $\widehat{  \mathbf{R}}_x$.

\vskip2mm
\noindent
{\bf Proposition 9.} {\it
Let $\chi$ be defined as in ~\cref{def_chisc}. 
For any $ x \in \supp( \mathrm{P} )$, $f \in {L}^2_{\mathrm{P}},$ $m>0,$
\[
\bigl\lVert \mathbf{R}_x - \widehat{  \mathbf{R}}_x \bigr\rVert_{L_{ \mathrm{P}}^2} \leq \mu(x)^{-1/2} \Bigl( (m-1) \lVert \chi \rVert_{\infty} \lVert \widehat{  \mathcal{  Q}} - \mathcal{  G} \rVert_{\infty} \left( 1+ \lVert \widehat{  \mathcal{  Q}} - \mathcal{  G} \rVert_{\infty}^{m-2} \right) + \lVert \chi - 1 \rVert_{\infty} \Bigr)
\]
and
\begin{multline*}
\Biggl( \int \Bigl\langle \mathbf{R}_x - \widehat{  \mathbf{R}}_x, f 
\Bigr\rangle^2_{L_{ \mathrm{P}}^2} \, \mathrm{d}\mathrm{P}(x) \Biggr)^{1/2} \leq \lVert f \rVert_{L_{ \mathrm{P}}^2} \Bigl( (m-1) \lVert \chi \rVert_{\infty} \lVert \widehat{  \mathcal{  Q}} - \mathcal{  G} \rVert_{\infty} \left( 1+ \lVert \widehat{  \mathcal{  Q}} - \mathcal{  G} \rVert_{\infty}^{m-2} \right)\\ + \lVert \chi - 1 \rVert_{\infty} \Bigr).
\end{multline*}
}

\vskip1mm
\noindent
For the proof we refer to Section~\ref{prof1}.

\vskip2mm
\noindent
The two above propositions link the quality of the approximation (of $\overline{K}_{2m}$ with $ \widehat{  K}_{2m}$ in Proposition 8  
and of $\mathbf{R}_x$ with $\widehat{  \mathbf{R}}_x$ in Proposition 9) to the quality of the approximation 
of the Gram operator $\mathcal G$ with $ \widehat{  \mathcal{  Q}}$.
In order to qualify the approximation error $ \lVert \widehat{  \mathcal{  Q}} - \mathcal{  G} \rVert_{\infty} $
we introduce a new intermediate (non completely observable) operator $\overline{\mathcal G}: \mathcal H \to \mathcal H$ defined as
\begin{align*}
\overline{\mathcal G}  v & = \frac{1}{n} \sum_{i=1}^n \mu(X_i)^{-1} \langle v, \phi_{K}(X_i) \rangle_{\mathcal H}\  \phi_{K}(X_i)\\
& = \frac{1}{n} \sum_{i=1}^n \langle v, \phi_{\overline K}(X_i) \rangle_{\mathcal H}\  \phi_{\overline K}(X_i),
\end{align*}
and we observe that the following result holds.

\vskip2mm
\noindent
{\bf Proposition 10.} {\it
Let $\chi$ be defined as in ~\cref{def_chisc}. 
It holds that 
\[
\lVert \widehat{  \mathcal  Q} - \mathcal{  G} \rVert_{\infty} 
\leq \lVert \overline{\mathcal G}  - \mathcal{  G} \rVert_{\infty} \left( 1+ \| \chi^2 - 1\|_{\infty} \right) + \| \chi^2 - 1\|_{\infty}.
\]
}

\vskip1mm
\noindent
For the proof we refer to Section~\ref{pf2}.

\vskip2mm
\noindent
Observe that, given $\phi_{\overline K}$, the operator $\overline{\mathcal G}$ is the empirical version of the Gram operator $\mathcal{  G}$.
Moreover, by definition, 
\[
\chi -1 = \left( \mu/ \widehat \mu\right)^{1/2}-1
\]
where, according to ~\cref{gA}, $\mu$ is the quadratic form associated to the Gram operator $\mathcal G_A$ and $\widehat \mu$
is its empirical version. 
Thus we conclude this section 
providing a result on the convergence of the empirical Gram estimator to the true one
that has been proved in ~\cite{tesi} and that appears in~\cite{giulini15}. \\
We first introduce some notation. 
Let $\mathcal H$ be a separable Hilbert space and let $\mathrm P$ be a probability distribution on $\mathcal H$. 
Let $\mathcal G: \mathcal H \to \mathcal H$ be the Gram operator
\[
\mathcal G v = \mathbb E_{X \sim \mathrm P} \left[ \langle v, X\rangle_{\mathcal H} X\right]
\]
and consider the empirical estimator 
\[
\widehat{\mathcal G} v = \frac{1}{n} \sum_{i=1}^n  \langle v, X_i\rangle_{\mathcal H} X_i
\]
where $X_1, \dots, X_n \in \mathcal H$ is an i.i.d. sample drawn according to $\mathrm P$. 
Let $\epsilon>0$ and let $\sigma >0$ be a threshold. Let
\[
\kappa  = \sup_{u \in \mathcal H} \frac{\displaystyle\mathbb E_{X \sim \mathrm P} \bigl[ \langle u , X \rangle_{\mathcal H}^4 \bigr]}{\mathbb E_{X \sim \mathrm P} \bigl[ \langle u, X \rangle_{\mathcal H}^2 \bigr]^2}<+\infty.
\]
Define
\begin{align*}
\zeta (t) & = \sqrt{ \frac{2.032 (\kappa-1)}{n} \Bigg( \frac{0.73 \ \mathbf{Tr}(\mathcal G)}{t}  + 4.35 + \log(\epsilon^{-1}) \Bigg)} + \sqrt{ \frac{98.5 \, \kappa \mathbf{Tr}(\mathcal G)}{ n t}}\\
\eta (t) & = \frac{ \zeta(\max \{ t, \sigma \} )}{1 - 4 \,  \zeta( \max \{ t, \sigma \} )} \\
\tau(t) & = { \frac{0.86 \max \| X_i\|_{\mathcal H}^4}{n (\kappa - 1)  \max \{ t, \sigma \}^2 } \Biggl( \frac{0.73 \ \mathbf{Tr}(\mathcal G)}{\max \{ t, \sigma \}}  + 4.35 + \log(\epsilon^{-1}) \Bigg)},
\end{align*}
where $ \mathbf{Tr}(\mathcal G)$ denotes the trace of $\mathcal G$.
The following proposition (proved in \cite{giulini15}) %\citep[proved in][]{giulini15} 
holds. 

\vskip2mm
\noindent
{\bf Proposition 11.} {\it
Let $\sigma >0$ be a threshold. 
With probability at least $1-2\epsilon$, for any $u \in \mathcal H$, $\| u\|_{\mathcal H}=1,$
\[
\left| \frac{\max \left\{ \langle \widehat{\mathcal G} u, u \rangle_{\mathcal H}  , \sigma \right\}}{\max \left\{ \langle {\mathcal G} u, u \rangle_{\mathcal H}  , \sigma \right\}} -1 \right| 
\leq  \eta \bigl( \langle {\mathcal G} u, u \rangle_{\mathcal H} \bigr) + \frac{ \tau\bigl( \langle {\mathcal G} u, u \rangle_{\mathcal H} \bigr)}{ \bigl[ 1 - \tau \bigl( \langle {\mathcal G} u, u \rangle_{\mathcal H} \bigr) \bigr]_+ \bigl[ 1 - \eta \bigl( \langle {\mathcal G} u, u \rangle_{\mathcal H} \bigr) \bigr]_+}.
\]
}

\vskip2mm
\noindent
As a consequence,

\vskip2mm
\noindent
{\bf Corollary 12.} {\it
With the same notation as before, with probability at least $1-2\epsilon$, 
\[
\| \mathcal G - \widehat{\mathcal G}\|_{\infty}
\leq  
\| \mathcal G\|_{\infty} \eta \bigl( \| \mathcal G\|_{\infty}  \bigr) + \frac{\sigma \ \tau\bigl( \sigma \bigr)}{ 
\bigl[ 1 - \tau \bigl( \sigma \bigr) \bigr]_+ \bigl[ 1 - \eta \bigl( \sigma \bigr) \bigr]_+} + \sigma.
\]
}

\vskip2mm
\noindent
As explained in~\cite{giulini15}, the threshold $\sigma$ can be chosen going to zero as the sample size $n$ goes to infinity. 
In particular, since $\sigma \tau(\sigma)$ behaves as $\frac{1}{n\sigma^2},$ the optimal value of $\sigma$ is of order $n^{-1/3}.$
As a consequence, according to the above results, 
we obtain a deviation bound in $n^{-1/3}$ for 
\[
\sup_{x,y \in \mathrm{supp}(\mathrm P)}  \bigl\lvert \widehat{  K}_{2m}(x,y) - \overline{K}_{2m}(x,y)\bigr\rvert.
\]
In order to get a deviation bound in $n^{-1/2}$ 
we have to use a more robust estimator for the Gram operator $\mathcal G$, defined in~\cite{giulini15} (see also \cite{tesi}),  
and to split the sample into two parts: the first part is used for the estimation of the kernel $\overline K$
and the other one for the construction of the estimator $\widehat {\mathcal Q}$.

\newpage
\section{Empirical Results}\label{sec3}

We present some results showing how the algorithm described in the previous sections simplifies the geometry of the problem and
in particular how it groups the points to cluster at the vertices of a simplex.
We first provide a toy example on synthetic data and then we test the algorithm in the setting of image analysis. 

\subsection{A First Example}

We consider an i.i.d. set $\left\{ X_1, \dots, X_n \right\}\subset \mathbb R^2$ of $n = 900$ points to cluster, whose configuration is shown in Figure ~\ref{fig0} and we fix the maximum number of classes $p = 7.$\\[1mm]

\begin{figure}[htbp]
\begin{center}
\includegraphics[height=70mm]{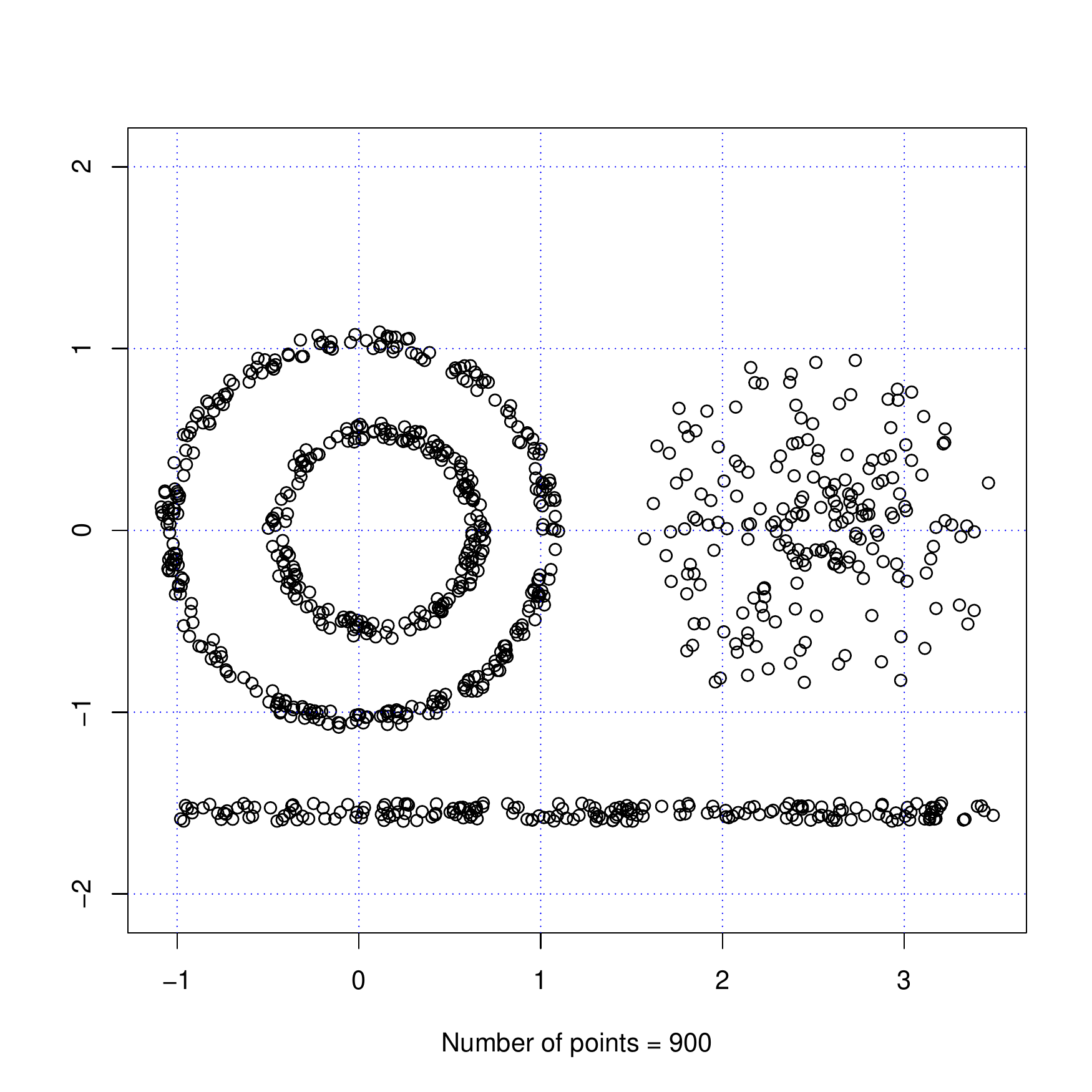}
\caption{The data configuration.}
\label{fig0}
\end{center}
\end{figure}

\vskip 2mm
\noindent
Figure ~\ref{figure1} shows that the new representation, induced by the change of kernel, 
groups the data points at the vertices of the simplex generated by the largest eigenvectors of the matrix $C$, defined in ~\cref{eqC}.
On the left we plot the projection of the simplex along the two first coordinates. 
This simple configuration allows us to compute the classification, 
including the number of clusters, using the straightforward greedy 
algorithm described in Remark 6.

\begin{figure}[htbp]
\begin{center}
\includegraphics[height=70mm]{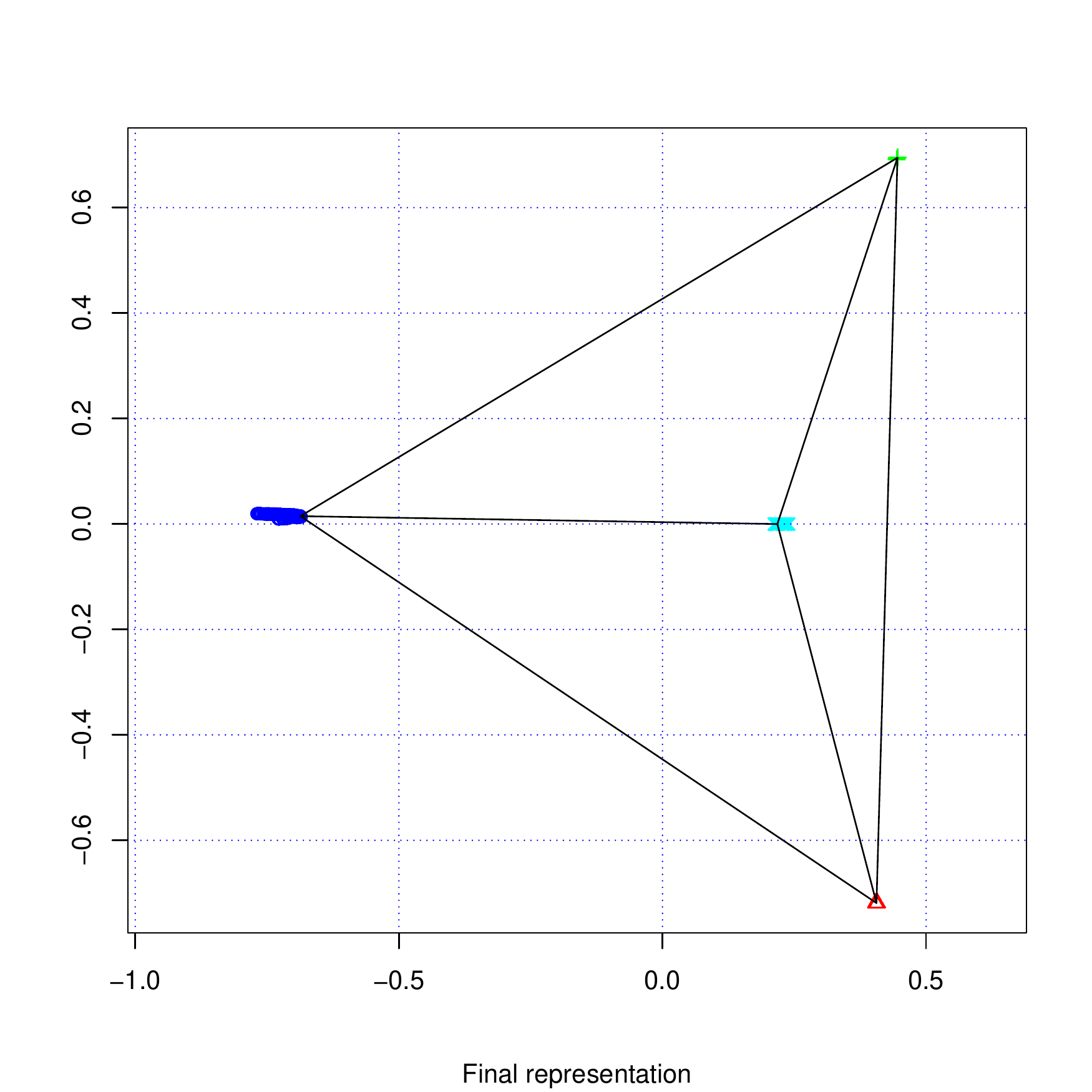}
\includegraphics[height=70mm]{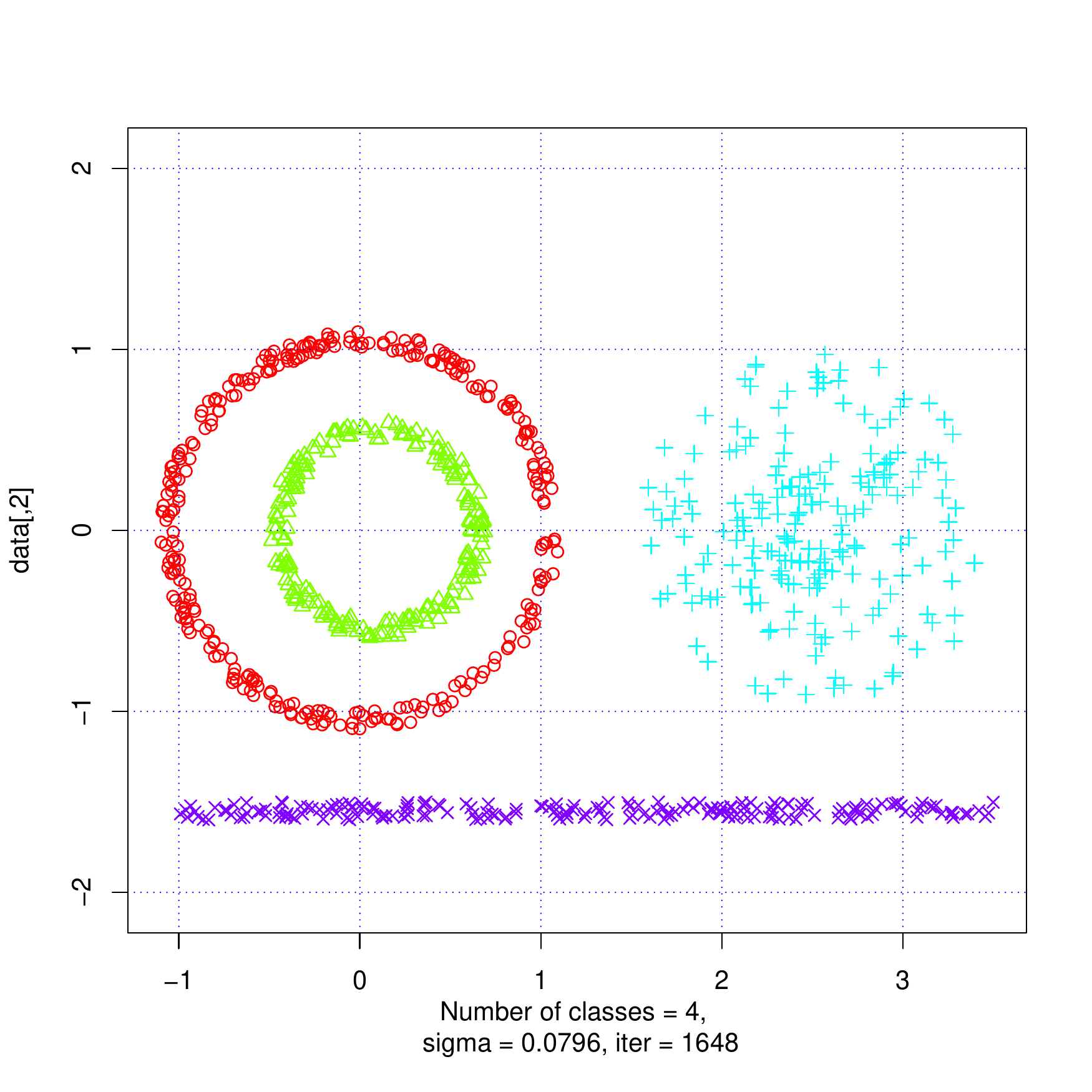}
\caption{On the left the simplex generated by the eigenvectors of $C$, on the right classification performed on $C$.}
\label{figure1}
\end{center}
\end{figure}

\vskip1mm
\noindent
In Figure ~\ref{fig_eigen} we plot the first eigenvalues
of $M$ in black (joined by a solid line),
the eigenvalues of its iteration $M^m$ 
in blue (joined by a dashed line)  
and the eigenvalues of the covariance matrix of the final 
representation (defined by ${C}$) in red (joined by a dash-dotted line).
We observe that the first eigenvalues of $M$ are close to one, while there is a remarkable gap between 
the eigenvalues of its iteration. 
In particular the size of the gap is larger once we have renormalized it using the matrix $C$. 
The number of iterations is automatically estimated and it is equal to $1648.$

\begin{figure}[htbp]
\begin{center}
\includegraphics[height=100mm]{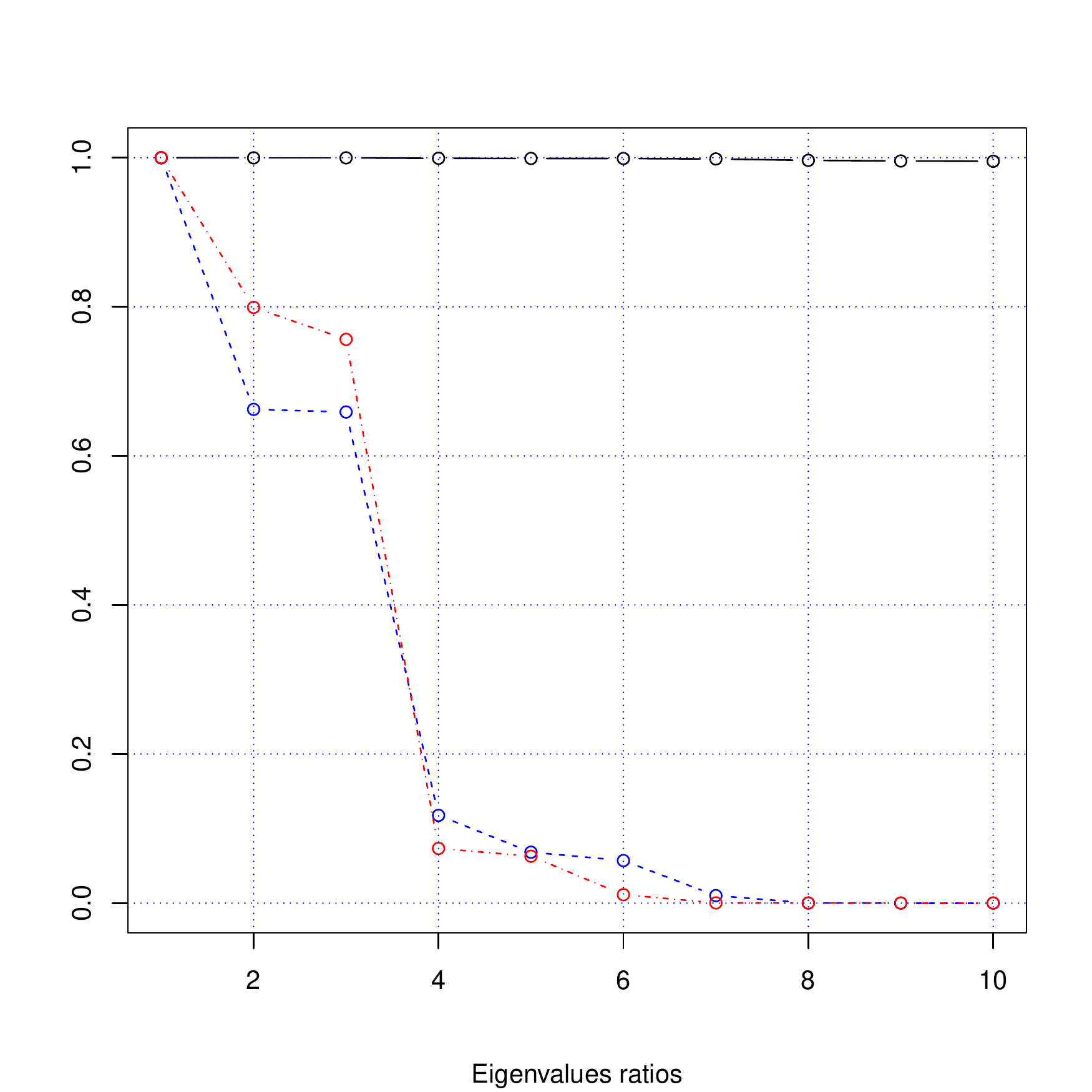}
\caption{In black (solid line) the eigenvalues of $M$, in blue (dashed line) those of  $M^m$ and in red (dash-dot line) the eigenvalues of $C$.}
\label{fig_eigen}
\end{center}
\end{figure}

\subsection{Some Perspectives on Invariant Shape Analysis} 

As we have already mentioned in the introduction, we now present a small example 
about image classification
showning
how our approach may lead to learn transformation-invariant representations 
from data sets containing small successive transformations of a same pattern. 
We briefly describe this approach to get a hint of its potential.
We consider two images (Figure ~\ref{foto_base}) and we create our patterns by translating a subwindow of a given size in each image repeatedly, using a translation vector smaller than the subwindow size.
In such a way we create a sample consisting of two classes of connected images, as shown in Figure ~\ref{seq_foto}.
This notion of translation cannot be grasped easily by a mathematical definition, because we do not translate a function 
but a window of observation. 
Hence in this case the translation depends on the image content and it may not be easy to model in any realistic situation. \\[1mm]
We present, on an example, a successful scenario in which we use 
first a change of representation in a reproducing kernel Hilbert space to better separate the two classes, and then spectral clustering to shrink each class to a tight blob. 
This suggests
that the so called kernel trick, introduced to better separate classes in the supervised learning framework of support vector machines (SVMs), also works in an unsupervised context. In this setting, we do not separate classes using hyperplanes, since we do not know the class labels which would be necessary to run a SVM, but, instead, we use spectral clustering to finish the work. 
\begin{figure}[htbp]
\begin{center}
\includegraphics[height=40mm]{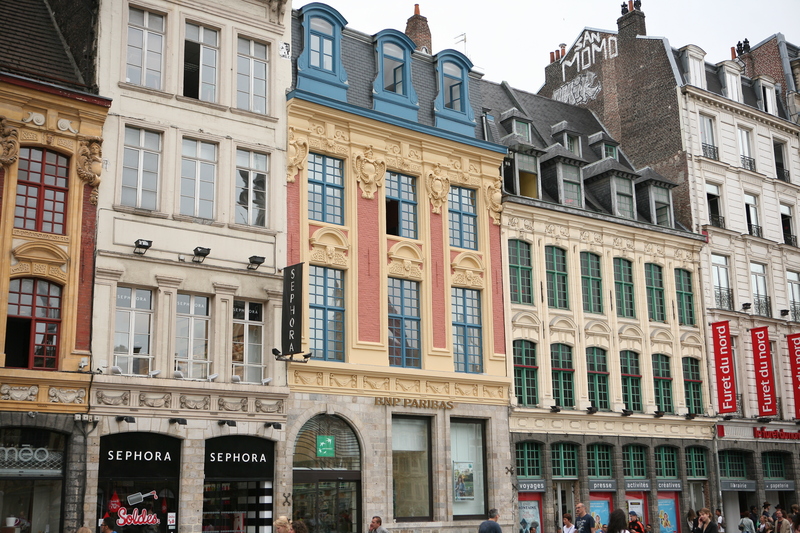}
\includegraphics[height=40mm]{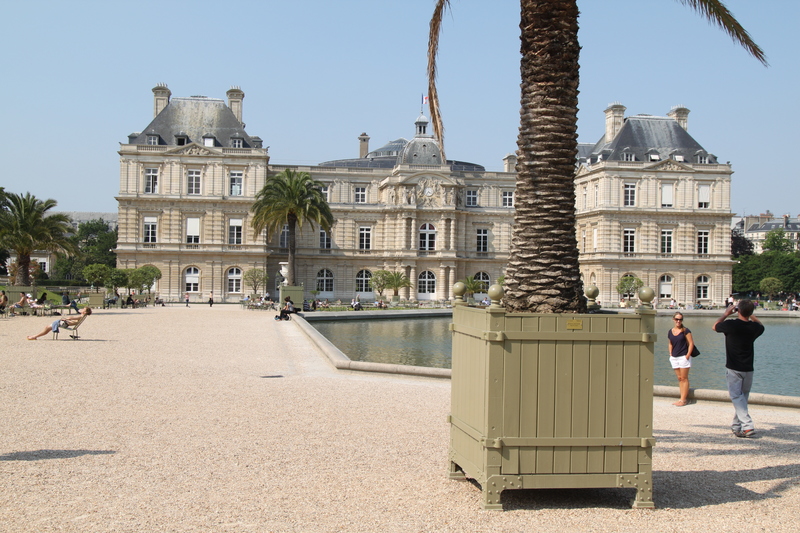}
\caption{The two original images.}
\label{foto_base}
\end{center}
\end{figure}

\begin{figure}[htbp]
\begin{center}
\includegraphics[width=0.1\textwidth]{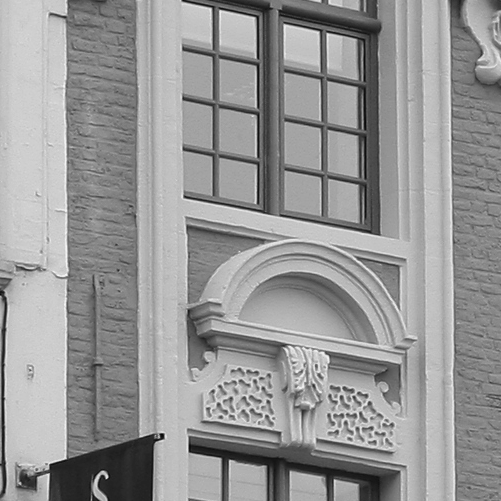}
\includegraphics[width=0.1\textwidth]{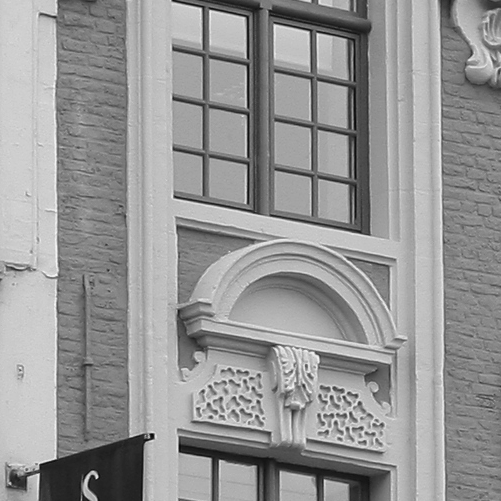}
\includegraphics[width=0.1\textwidth]{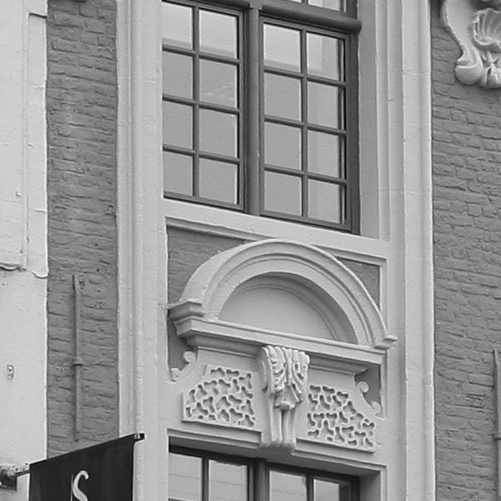}
$\cdots$ 
\includegraphics[width=0.1\textwidth]{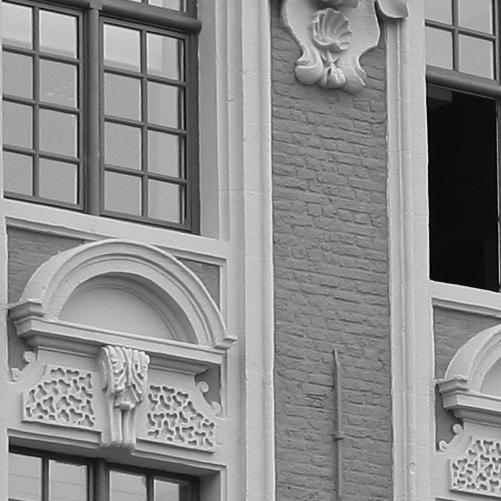}
\includegraphics[width=0.1\textwidth]{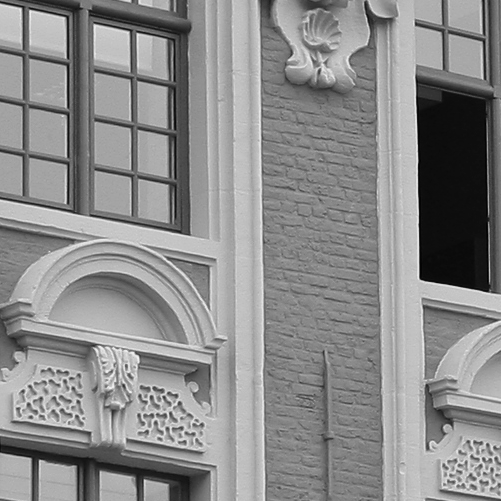}
$\cdots$ 
\includegraphics[width=0.1\textwidth]{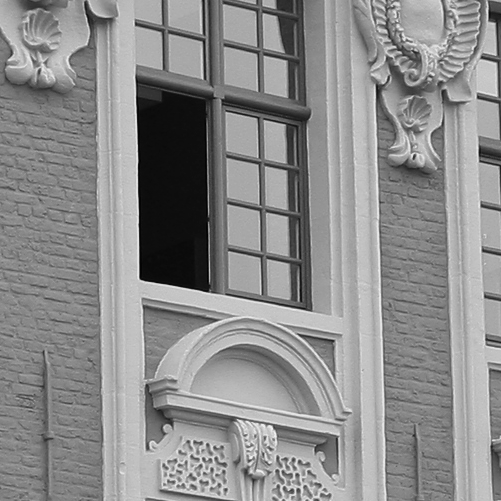}
\includegraphics[width=0.1\textwidth]{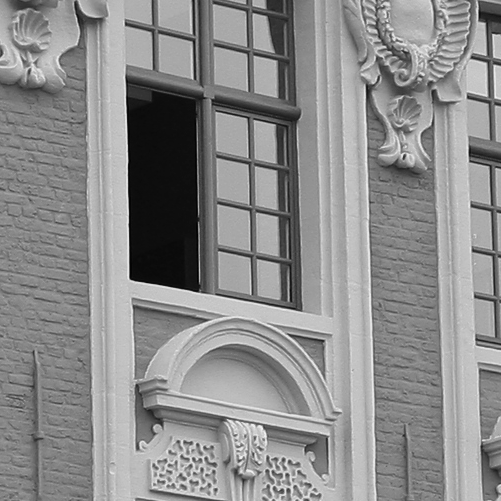}
\includegraphics[width=0.1\textwidth]{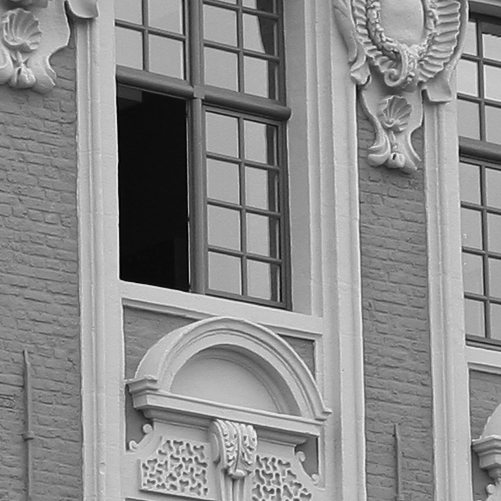}\\[3mm]
\includegraphics[width=0.1\textwidth]{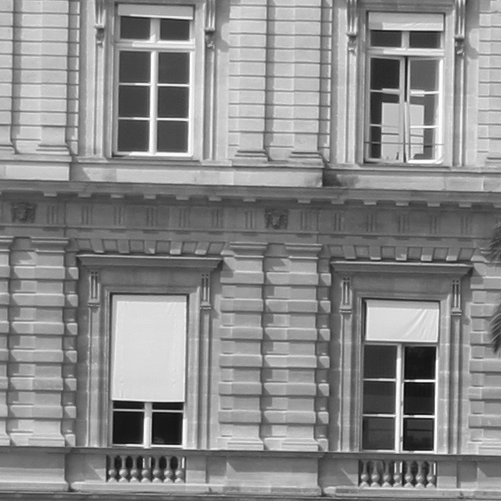}
\includegraphics[width=0.1\textwidth]{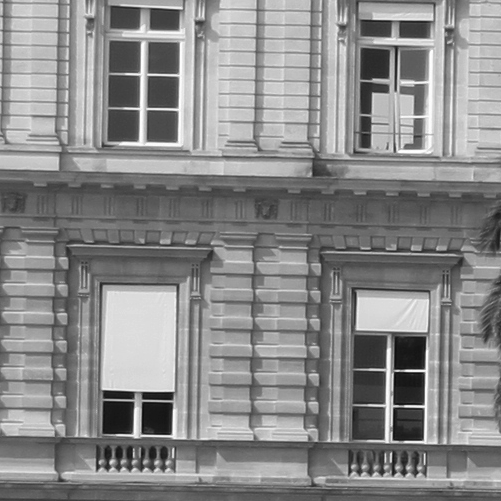}
\includegraphics[width=0.1\textwidth]{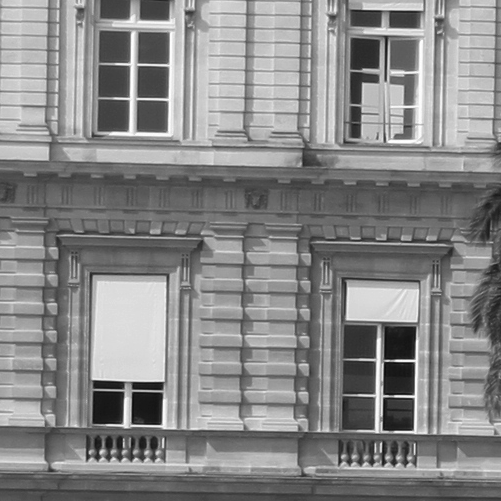}
$\cdots$ 
\includegraphics[width=0.1\textwidth]{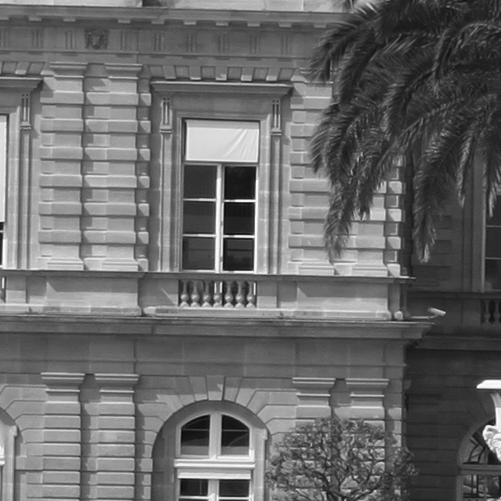}
\includegraphics[width=0.1\textwidth]{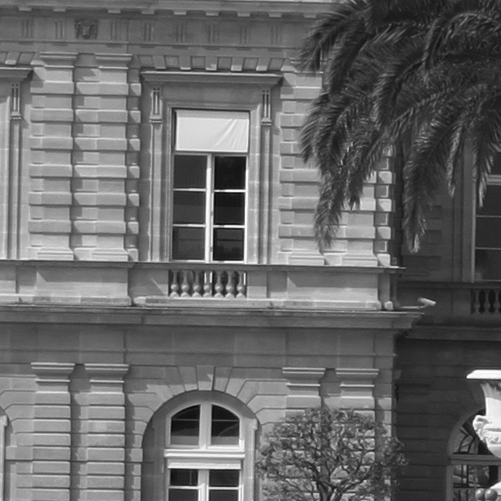}
$\cdots$ 
\includegraphics[width=0.1\textwidth]{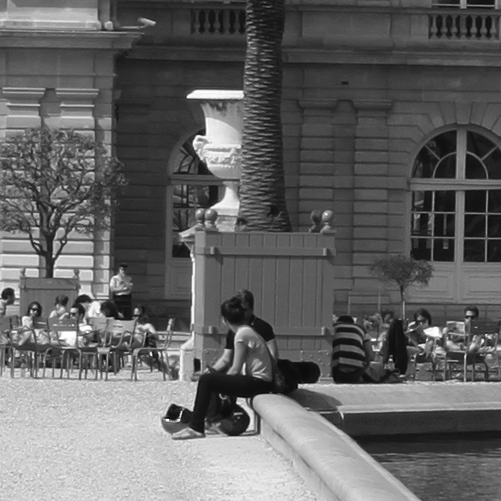}
\includegraphics[width=0.1\textwidth]{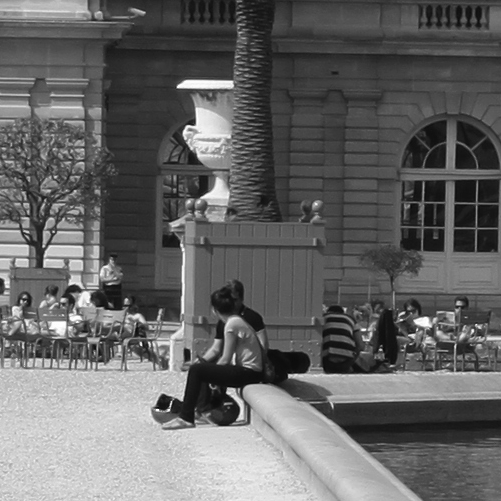}
\includegraphics[width=0.1\textwidth]{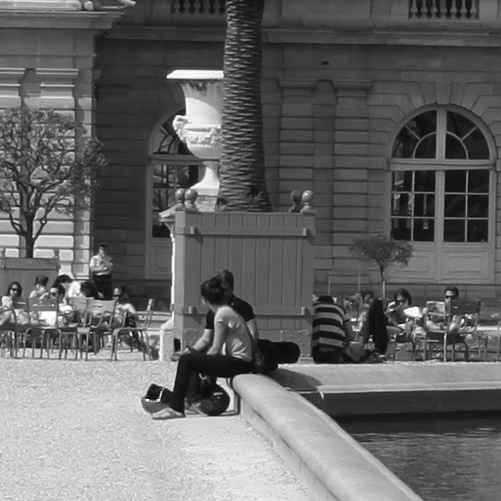}
\caption{Our sample consisting of two classes of connected images, the first sequence is obtained with a horizontal translation, the second one with a diagonal translation.}
\label{seq_foto}
\end{center}
\end{figure}

\vskip 2mm
\noindent
Let $X_1, \dots, X_n$ be the sample of images shown in Figure ~\ref{seq_foto}.
Each photo is represented as a matrix whose entries are the gray values of the corresponding pixels. 
We apply twice the change of representation described by the change of kernel in order to better separate clusters. 
We first consider the reproducing kernel Hilbert space $\mathcal H_1$ defined by
\[
k_1( x, y ) = \exp \left(- \beta_1 \| x-y\|^2 \right)
\]
and then the reproducing kernel Hilbert space $\mathcal H_2$ defined by 
\begin{align*}
k_2( x, y )  = \exp \left(- \beta_2 \| x-y\|_{\mathcal H_1}^2 \right) & =  \exp \Bigl(- 2\beta_2 \bigl( 1-k_1(x,y) \bigr) \Bigr)\\
&  = \exp \left[ -2\beta_2 \left( 1 - \exp \left( -\beta_1 \| x - y \|^2 \right) \right)\right]
\end{align*}
where $\beta_1,\beta_2>0$ are obtained as described in Section~\ref{choice_beta}. 
Define the new kernel
\[
K(x,y)=  \exp \left(- \beta \| x-y\|_{\mathcal H_2}^2 \right),
\]
where the parameter $\beta>0$ is chosen again as in Section~\ref{choice_beta}, 
and apply the algorithm described in Section~\ref{impa}.

\vskip2mm
\noindent
In Figure ~\ref{rep_images} we compare the representation of the images in the initial space and in the space $\mathcal H_2.$ 
On the left we present the projection of the sample onto 
the space spanned by the first two largest eigenvectors of the matrix of inner products between images 
$\langle X_i, X_j\rangle. $
On the right we plot the projection onto the space spanned by the two largest eigenvectors of the matrix of inner products $ k_2(X_i, X_j)$ in $\mathcal H_2$. 
We observe that in the first representation the two classes intersect each other while in the second one, after the change of representation, they are already separated.  \\[1mm]
To conclude, Figure ~\ref{fin_rep} shows the final representation. 
Here the data points are projected onto the space spanned by the two largest eigenvectors of the matrix 
$M^m.$ In this case the number of iteration $m$ is of order of $30.000$.

\begin{figure}[htbp]
\begin{center}
\includegraphics[height=50mm]{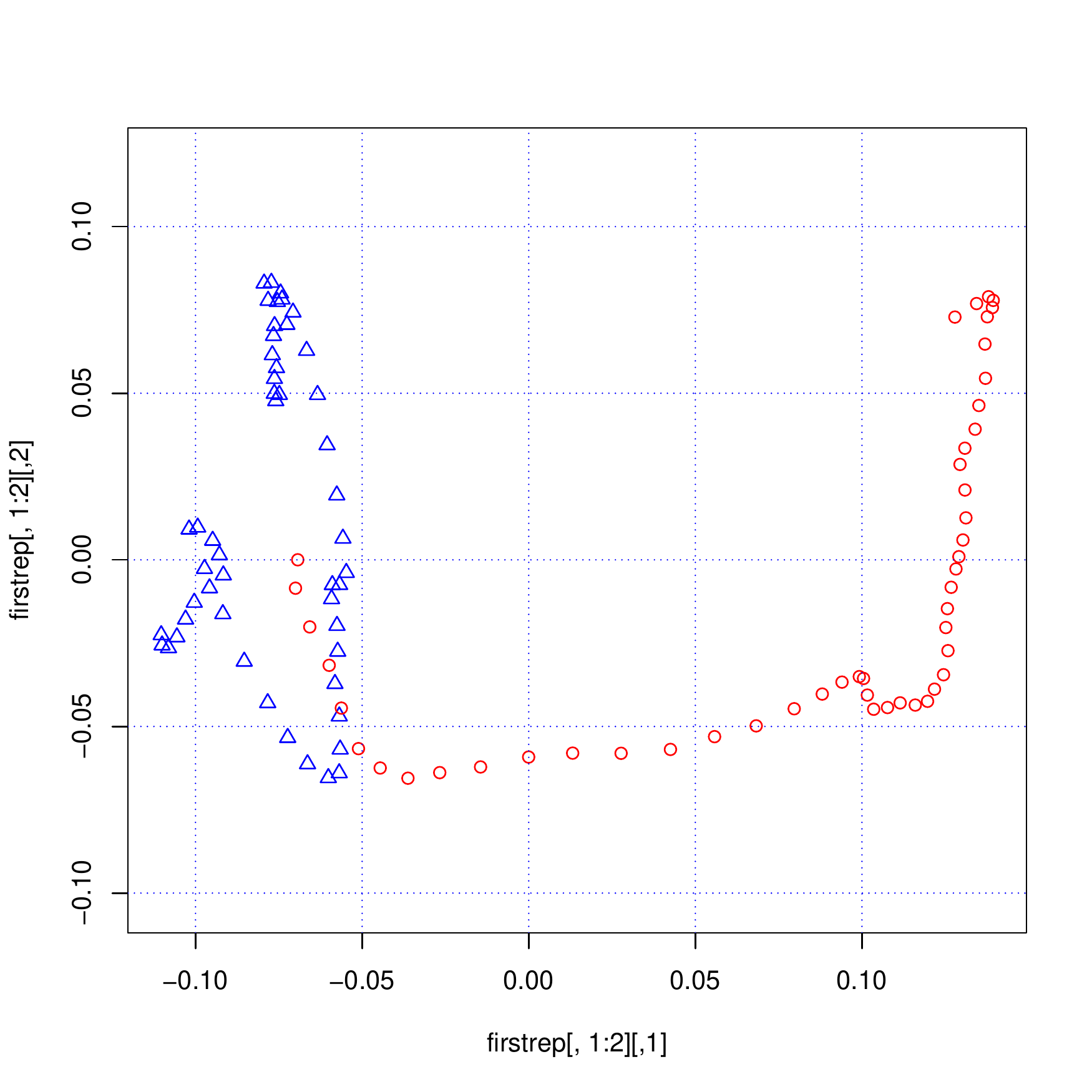}
\includegraphics[height=50mm]{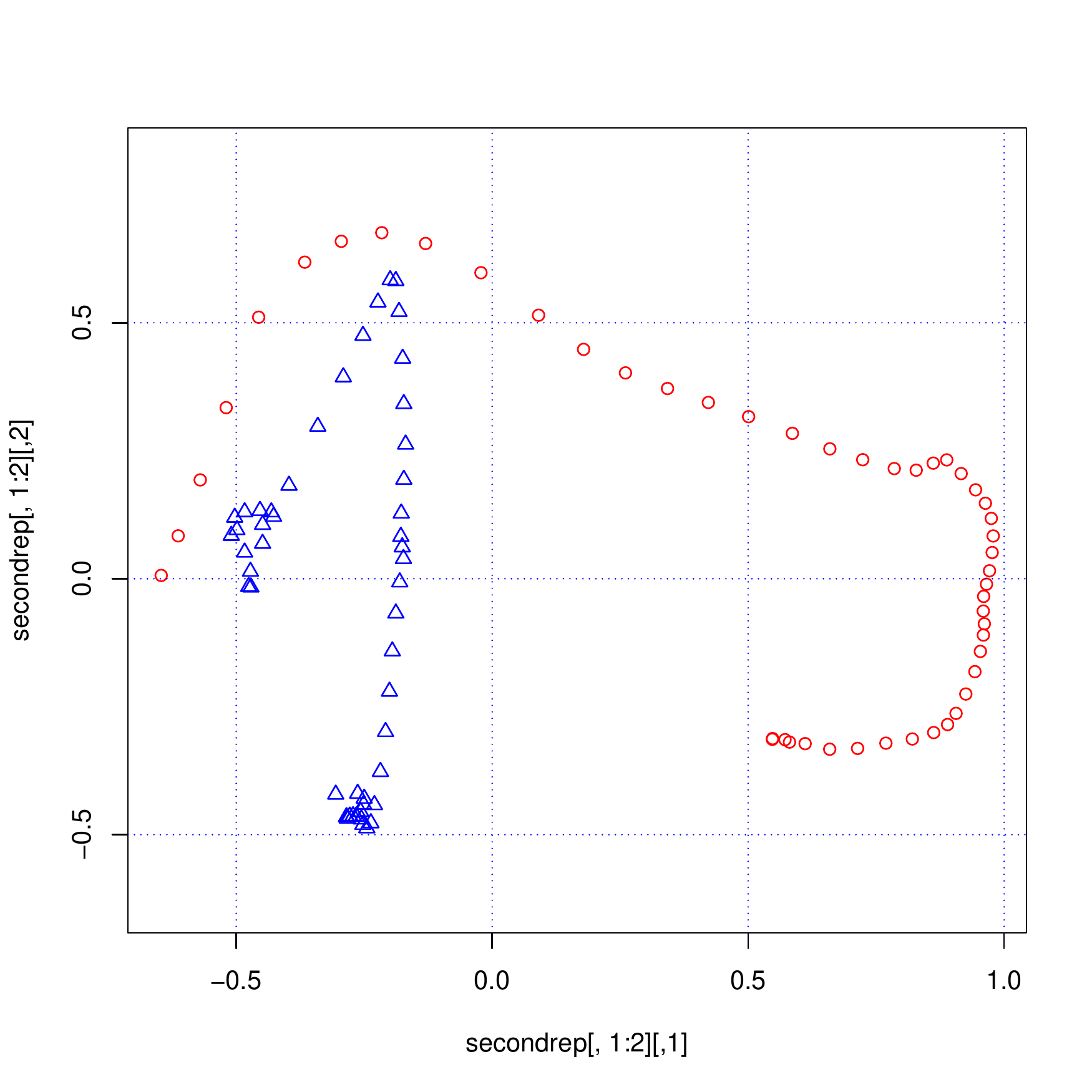}
\caption{On the left the projection onto the space spanned by the two largest eigenvectors of $ \langle X_i, X_j\rangle$, on the right the projection onto the space spanned by the two largest eigenvectors of $k_2(X_i, X_j)$.}
\label{rep_images}
\end{center}
\end{figure}

\begin{figure}[htbp]
\begin{center}
\includegraphics[height=50mm]{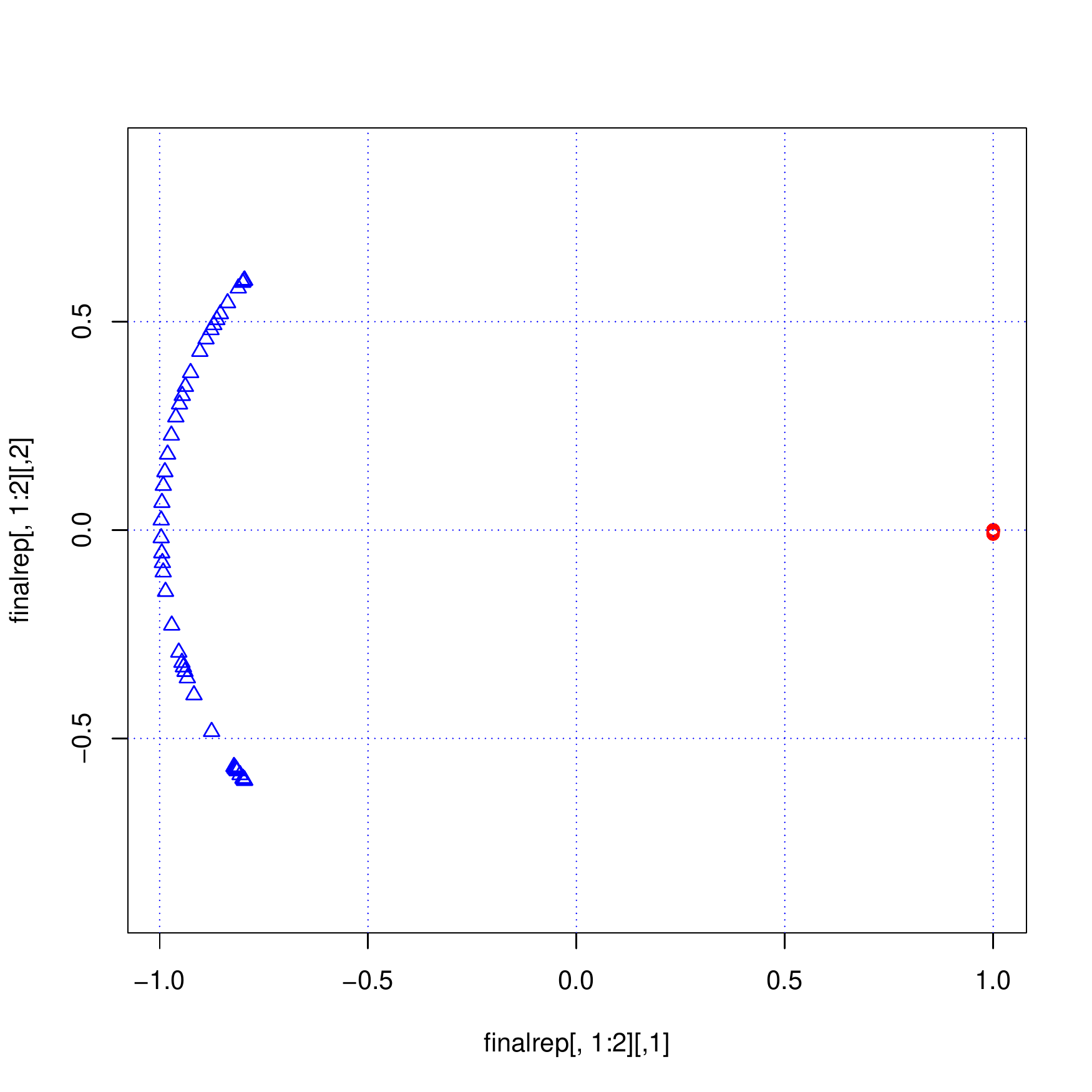}
\caption{The projection onto the space spanned by the two largest eigenvectors of $M^m.$}
\label{fin_rep}
\end{center}
\end{figure}

\newpage

\section{Proofs}\label{proofs}

We introduce a technical result that is useful in the following. 
Let $\mathcal G$ be the Gram operator defined in ~\cref{defGop} such that $\| \mathcal G\|_\infty=1$
and let $\widehat{\mathcal Q}$
be its estimator introduced in Section~\ref{emp_algo}.

\vskip2mm
\noindent
{\bf Lemma 13.} {\it
For any $m>0$
\begin{align*}
\lVert \widehat{  \mathcal{  Q}}^m - \mathcal{  G}^m \lVert_{\infty} 
& \leq 
\Bigl(  1 + \lVert \widehat{  \mathcal{  Q}} - \mathcal{  G} \rVert_{\infty} \Bigr)^m - 1  \\
& \leq m \lVert \widehat{  \mathcal{  Q}} - \mathcal{  G} \rVert_{\infty}  \Bigl( 1 + \lVert 
\widehat{  \mathcal{  Q}} - \mathcal{  G} \rVert_{\infty} \Bigr)^{m-1}.
\end{align*}
}

\begin{proof}
Since 
\[ 
\widehat{  \mathcal{  Q}}^m - \mathcal{  G}^m = \sum_{k=0}^{m-1} \widehat{  \mathcal{  Q}}^{k} 
( \widehat{  \mathcal{  Q}} - \mathcal{  G} ) \mathcal{  G}^{m-k-1}, 
\] 
using the fact that $\| \mathcal G\|_\infty=1$ we get
\[ 
\lVert \widehat{  \mathcal{  Q}}^m - \mathcal{  G}^m \rVert_{\infty} \leq \lVert \widehat{  \mathcal{  Q}} - \mathcal{  G} \rVert_{\infty} \sum_{k=0}^{m-1} 
\lVert \widehat{  \mathcal{  Q}} \rVert_{\infty}^k. 
\] 
Hence, as $\lVert \widehat{  \mathcal{  Q}} \rVert_{\infty} \leq 1 + \lVert \widehat{  \mathcal{  Q}} - \mathcal  G \rVert_{\infty}$, we conclude 
\begin{multline*}
\lVert \widehat{  \mathcal{  Q}}^m - \mathcal  G^m  \rVert_{\infty} 
 \leq  \Bigl(  1 + \lVert \widehat{  \mathcal{  Q}} - \mathcal{  G} \rVert_{\infty} \Bigr)^m - 1 \\ 
= \lVert \widehat{  \mathcal{  Q}} - \mathcal{  G} \rVert_{\infty} \sum_{k=0}^{m-1} \bigl( 1 + \lVert \widehat{  \mathcal{  Q}} - \mathcal{  G} \rVert_{\infty} \bigr)^k
\leq m \lVert \widehat{  \mathcal{  Q}} - \mathcal{  G} \rVert_{\infty} \bigl( 1 + \lVert \widehat{  \mathcal{  Q}} - \mathcal{  G} \rVert_{\infty} \bigr)^{m-1}. 
\end{multline*}
$\blacksquare$
\end{proof}

\vskip2mm
\noindent
We also introduce the intermediate operator $\widetilde{\mathcal G} : \mathcal H \to \mathcal H$
\[
\widetilde{\mathcal G} u = \int \left\langle u,  \phi_{\widehat K}(z) \right\rangle_{\mathcal H} \ \phi_{\widehat K}(z) \, \mathrm d \mathrm P(z)
\]
where, according to the notation of Section~\ref{sec2}, $\phi_{\widehat K} (x) = \chi(x) \phi_{\overline K}(x)$ and $\chi(x) = \big( \mu(x)/ \widehat \mu (x)\big)^{1/2}.$

\subsection{Proof of Proposition 8}\label{pf1}

To prove the first inequality, we define 
\[
\mathcal{  E}(x, y) = \bigl\lvert \widehat{  K}_{2m}(x,y) - \overline{K}_{2m}(x,y)\bigr\rvert
\]
and we observe that, according to Lemma 13,  
it is sufficient to show that 
\begin{equation}
\label{b1}
\mathcal{  E}(x,y) \leq \frac{\displaystyle \max \{1, \lVert \chi \rVert_{\infty}\}^2}{\mu(x)^{1/2} \mu(y)^{1/2}} \Bigl( \lVert \widehat{  \mathcal{  Q}}^{2m-1} - \mathcal{  G}^{2m-1} \rVert_{\infty} + 2 \lVert \chi - 1 \rVert_{\infty} \Bigr).
\end{equation}
By definition, 
\begin{multline*}
\mathcal{  E}(x,y) = \bigl\lvert \left\langle \widehat{  \mathcal{  Q}}^{2m-1}   \phi_{\widehat K}(x),  \phi_{\widehat K}(y) \right\rangle_{\mathcal{  H}} - \left\langle \mathcal{  G}^{2m-1} \phi_{\overline K}(x), \phi_{\overline K}(y) 
\right\rangle_{\mathcal{  H}} \bigr\rvert \\ 
\leq \bigl\lvert \left\langle (\widehat{  \mathcal{  Q}}^{2m-1} - \mathcal{  G}^{2m-1} ) \phi_{\widehat K}(x), \phi_{\widehat K}(y) \right\rangle_{\mathcal{  H}} \bigr\rvert + \bigl\lvert \left\langle 
\mathcal{  G}^{2m-1} \bigl( \phi_{\widehat K}(x) - \phi_{\overline K}(x) \bigr), \phi_{\widehat K}(y) \right\rangle_{\mathcal{  H}} \bigr\rvert \\ + \left\langle \mathcal{  G}^{2m-1} \phi_{\overline K}(x), \bigl( \phi_{\widehat K}(y) - \phi_{\overline K}(y) \bigr)\right\rangle_{\mathcal{  H}} \bigr\rvert.
\end{multline*}
Recalling the definition of $\phi_{\widehat K},$ 
we get
\[
\mathcal{  E}(x,y) \leq \lVert \widehat{  \mathcal{  Q}}^{2m-1} - \mathcal{  G}^{2m-1} \rVert_{\infty} \lVert 
\chi \rVert_{\infty}^2 \lVert \phi_{\overline K}(x) \rVert_{\mathcal{  H}} \lVert \phi_{\overline K}(y) 
\rVert_{\mathcal{  H}} 
+ \lVert \chi - 1 \rVert_{\infty} \bigl( 
\lVert \chi \rVert_{\infty} + 1 \bigr) \lVert \phi_{\overline K}(x) \rVert_{\mathcal{  H}} \lVert \phi_{\overline K}(y) \rVert_{\mathcal{  H}},
\]
where, since $K(x,x) =1$, 
\[
 \lVert \phi_{\overline K}(x) \rVert_{\mathcal{  H}}^2 = \overline K(x,x) = \frac{K(x,x)}{\mu(x)^{1/2} \mu(x)^{1/2}} = \frac{1}{\mu(x)}. 
\]
This proves ~\cref{b1}. 
To prove the second bound we define 
\[
\mathcal{  E}(x) = \Biggl( \int \mathcal{  E}(x,y)^2 \, \mathrm{d}\mathrm{P}(y) \Biggr)^{1/2}
\]
so that using Lemma 13  
again, it is sufficient to show that 
\[
\mathcal{  E}(x) \leq \frac{\displaystyle \max \{ 1, \lVert \chi \rVert_{\infty}\}^2}{\mu(x)^{1/2}} \Bigl( \lVert \widehat{  \mathcal{  Q}}^{2m-1} - \mathcal{  G}^{2m-1} \rVert_{\infty} + 2 \lVert \chi - 1 \rVert_{\infty} \Bigr).
\]
Observe that 
\begin{multline*}
\mathcal{  E}(x) = \Biggl( \int \Bigl( \left\langle \widehat{  \mathcal{  Q}}^{2m-1} \phi_{\widehat K}(x), 
\phi_{\widehat K}(y) \right\rangle_{\mathcal{  H}} - \left\langle \mathcal{  G}^{2m-1} \phi_{\overline K}(x), 
\phi_{\overline K}(y) \right\rangle_{\mathcal{  H}} \Bigr)^2 \, \mathrm{d}\mathrm{P}(y) \Biggr)^{1/2} 
\\ 
\leq \Biggl( \int \left\langle 
\bigl( \widehat{  \mathcal{  Q}}^{2m-1} - \mathcal{  G}^{2m-1} \bigr)   \phi_{\widehat K}(x), 
\phi_{\widehat K}(y) \right\rangle_{\mathcal{  H}}^2 \, \mathrm{d}\mathrm{P}(y) \Biggr)^{1/2} 
\\ \qquad \qquad \qquad+ \Biggl( \int \bigl\langle \mathcal{  G}^{2m-1} \bigl( 
 \phi_{\widehat K}(x) - \phi_{\overline K}(x) \bigr), \phi_{\widehat K}(y) \bigr\rangle_{\mathcal{  H}}^2 \, 
\mathrm{d}\mathrm{P}(y) \Biggr)^{1/2} \\ + 
\Biggl( \int \left\langle \mathcal{  G}^{2m-1} \phi_{\overline K}(x), \phi_{\widehat K}(y) - \phi_{\overline K}(y) 
\right\rangle_{\mathcal{  H}}^2 \, \mathrm{d}\mathrm{P}(y) \Biggr)^{1/2}. 
\end{multline*}
By definition of $\widetilde{\mathcal G}$ we get 
\begin{multline*}
\mathcal{  E}(x)  
\leq \bigl\lVert  \widetilde {\mathcal G}^{1/2} \bigl( \widehat{  \mathcal{  Q}}^{2m-1} -  \mathcal{  G}^{2m-1} \bigr) \phi_{\widehat K}(x) \bigr\rVert_{\mathcal{  H}} + \bigl\lVert  \widetilde {\mathcal G}^{1/2} \mathcal{  G}^{2m-1} \bigl( \phi_{\widehat K}(x) - \phi(x) \bigr) \rVert_{\mathcal{  H}} \\ 
+ \lVert \chi - 1 \rVert_{\infty} \bigl\lVert \mathcal{  G}^{2m-1/2} \phi(x) \bigr\rVert_{\mathcal{  H}}.
\end{multline*}
Thus using the fact that $\| \mathcal G\|_{\infty} =1$
and that, for any $u \in \mathcal H$,
\[
\|  \widetilde {\mathcal G}^{1/2} u\|_{\mathcal  H}^2 =  \langle  \widetilde {\mathcal G} u, u \rangle_{\mathcal  H} = 
\int \langle u , \phi_{\widehat K}(y) \rangle_{\mathcal  H}^2 \ \mathrm d \mathrm P(y)
\leq \lVert \chi \rVert_{\infty}^2 \langle \mathcal{  G} u, u \rangle, 
\]
we conclude.

\subsection{Proof of Proposition 9}\label{prof1} 

In order to prove the two inequalities we need to introduce some preliminary results.
Consider the operator $\mathcal{  S} : \mathcal{  H} \to L_{ \mathrm{P}}^2$ 
\[
\mathcal{  S}(u) : x \mapsto \langle u, \phi_{\overline K} (x) \rangle_{\mathcal{  H}}.
\]
Introduce the operator $\mathbf{G} :L_{ \mathrm{P}}^2  \to L_{ \mathrm{P}}^2$
\[
\mathbf{G}(f)(x) =  \mathcal{  S} \mathcal{  S}^* f(x)= \int \overline K(x, y) f(y) \, \mathrm{d}\mathrm{P}(y)
\]
and observe that the Gram operator $\mathcal G$ rewrites as 
\[
\mathcal{  G} u = \mathcal{  S}^* \mathcal{  S} u =\int \langle u, \phi_{\overline K}(z) \rangle_{\mathcal H}\  \phi_{\overline K}(z) \ \mathrm d \mathrm P(z) .
\]
Moreover $\| \mathbf{G} \|_{\infty} = \| \mathcal  G \|_{\infty} = 1.$

\vskip2mm
\noindent
{\bf Lemma 14.} {\it 
We have 
\[
\im(\mathbf{G}) = \mathcal{  S} \bigl( \im(\mathcal{  G}) \bigr) = \mathcal{  S} \bigl( \im(   \widetilde {\mathcal G} ) \bigr).
\]
}

\begin{proof}
We observe that, since $\mathbf{G}$ is symmetric, $\im(\mathbf{G}) = \im(\mathbf{G}^2)$, 
where $\mathbf{G}^2 = \mathcal{  S} \mathcal{  G} \mathcal{  S}^*$. Thus
$\im(\mathbf{G}) = \im(\mathbf{G}^2) \subset \mathcal{  S}\bigl(\im(\mathcal{  G})\bigr)$. Since $\mathcal{  S} \mathcal{  G} = \mathbf{G} \mathcal{  S}$, we conclude that $\mathcal{  S} \bigl( \im(\mathcal{  G}) \bigr) \subset \im \mathbf{G}$. \\[1mm]
To prove the second identity, we remark that, since $\chi(x)>0$, 
\begin{align*}
\Kern( \widetilde{\mathcal G} ) &= \biggl\{ u \in \mathcal{  H}, \int \langle u,\phi_{\widehat K}(x) \rangle^2 \, \mathrm{d}\mathrm{P}(x) = 0 \biggr\} \\
& = \biggl\{ u \in \mathcal{  H}, \int \langle u, \phi(x) \rangle^2 \, \mathrm{d}\mathrm{P}(x) = 0 \biggr\} = \Kern(\mathcal{  G}).
\end{align*}
Consequently, since $\im(\mathcal{  G}) = \Kern(\mathcal{  G})^{\perp}$, we get 
\begin{align*}
\cspan \Bigl( \phi \bigl( \supp(P) \bigr)  \Bigr) & = \Kern(\mathcal{  G})^{\perp} = \im(\mathcal{  G})  = \im(\widetilde{\mathcal G} ) = \cspan \Bigl( \phi_{\widehat K} \bigl( \supp(P) \bigr) \Bigr).
\end{align*}
$\blacksquare$
\end{proof}

\vskip 2mm
\noindent
Therefore any $f \in \im(\mathbf{G})$ is of the form $f = \mathcal{  S}u$, with $u \in \im(\mathcal{  G}),$ so that we can estimate
\[
\lVert f \rVert_{L_{ \mathrm{P}}^2}^2 = \langle \mathcal  G u, u \rangle_{\mathcal H}
\]
by  $\langle \widehat{  \mathcal{  Q}} u, u \rangle_{\mathcal{  H}}$. The estimation error is bounded as described in the following 
lemma. 

\vskip2mm
\noindent
{\bf Lemma 15.} {\it 
For any $u, v \in \mathcal{  H}$, 
\[ 
\bigl\lvert \langle \mathcal{  S} u , \mathcal{  S} v \rangle_{L_{ \mathrm{P}}^2} 
- \langle \widehat{  \mathcal{  Q}} u, v \rangle_{\mathcal{  H}} \bigr\rvert \leq \lVert 
\mathcal{  G} - \widehat{  \mathcal{  Q}} \rVert_{\infty} \lVert u \rVert_{\mathcal{  H}} 
\lVert v \rVert_{\mathcal{  H}}.
\] 
In particular, 
\[ 
\bigl\lvert \lVert \mathcal{  S}u \rVert^2_{L_{ \mathrm{P}}^2} - \langle 
\widehat{  \mathcal{  Q}} u , u \rangle_{\mathcal{  H}}
\bigr\rvert \leq \lVert \mathcal{  G} - \widehat{  \mathcal{  Q}} \rVert_{\infty} \lVert u \rVert^2_{\mathcal{  H}}.
\] 
}

\begin{proof}
It is sufficient to observe that $ \langle \mathcal{  S} u, \mathcal{  S} v \rangle_{L_{ \mathrm{P}}^2} =  \langle \mathcal{  G} u, v \rangle_{\mathcal H}$.
$\blacksquare$
\end{proof}

\vskip2mm
\noindent
We now observe that 
\begin{equation}\label{R:imG} 
\aligned
\mathbf{R}_x & = \mathcal{  S} \mathcal{  G}^{m-1} {\phi_{\overline K}}(x)\in \im(\mathbf{G})\\
\widehat{  \mathbf{R}}_x  & = \mathcal{  S} \widehat{  \mathcal{  Q}}^{m-1} \phi_{\widehat K}(x) \in \im(\mathbf{G}) \quad \text{almost surely}.
\endaligned
\end{equation} 
Indeed, by definition,
\begin{align*}
\widehat{  \mathbf{R}}_x  = \mathcal{  S} \widehat{  \mathcal{  Q}}^{m-1} \phi_{\widehat K}(x) \quad \text{and} \quad 
\mathbf{R}_x  = \mathcal{  S} \mathcal{  G}^{m-1} {\phi}_{\overline K}(x).
\end{align*}  
Hence, since 
\[
\im(\mathbf{G}) = \mathcal{  S} \bigl( \im ( \mathcal{  G} ) \bigr) = \mathcal{  S} \Bigl( \cspan  \phi_{\overline K} \bigl( \supp ( \mathrm{P} ) \bigr) \Bigr),
\]
we conclude that 
$\mathbf{R}_x \in \im(\mathbf{G})$. Moreover, since 
\begin{align*}
\Vect \bigl\{ \phi_{\widehat K}(X_n), \dots,  \phi_{\widehat K}(X_{2n})\bigr\} & = \Vect \bigl\{ \phi_{\overline K}(X_n), \dots, \phi_{\overline K}(X_{2n}) \bigr\} \\
& \subset \cspan \Bigl( \phi_{\overline K} \bigl( \supp ( \mathrm{P} ) \bigr) \Bigr) \quad \text{almost surely,}
\end{align*}
it is also true that $\widehat{  \mathbf{R}}_x \in \im( \mathbf{G} )$. 

\vskip2mm
\noindent
We can now prove the two bounds presented in the proposition. Define
\begin{align*}
\mathcal{  E}_r(x) & = \bigl\lVert \mathbf{R}_x - \widehat{  \mathbf{R}}_x \bigr\rVert_{L_{ \mathrm{P}}^2}\\
\mathcal{  E}_c(f) & = \Biggl( \int \Bigl\langle \mathbf{R}_x - \widehat{  \mathbf{R}}_x, f \Bigr\rangle^2_{L_{ \mathrm{P}}^2} \, \mathrm{d}\mathrm{P}(x) \Biggr)^{1/2}.
\end{align*}
According to Lemma 13, 
it is sufficient to show that, 
for any $ x \in \supp( \mathrm{P} )$, $f \in {L}^2_{\mathrm{P}},$
\begin{align*}
\mathcal{  E}_r(x) & \leq \mu(x)^{-1/2} \Bigl( \lVert \chi \rVert_{\infty} \lVert \widehat{  \mathcal{  Q}}^{m-1} - \mathcal{  G}^{m-1} \rVert_{\infty} + \lVert \chi - 1 \rVert_{\infty} \Bigr) \\
\mathcal{  E}_c(f) & \leq \lVert f \rVert_{L_{ \mathrm{P}}^2} \Bigl( \lVert \chi \rVert_{
\infty} \lVert \widehat{  \mathcal{  Q}}^{m-1} - \mathcal{  G}^{m-1} \lVert_{\infty} + \lVert \chi - 1 \rVert_{\infty} \Bigr).
\end{align*}

\vskip2mm
\noindent
In order to prove the first inequality, we observe that 
\begin{align*}
\mathcal{  E}_r(x) & = 
\bigl\lVert \mathcal{  S} \widehat{  \mathcal{  Q}}^{m-1} \phi_{\widehat K}(x) - \mathcal{  S} \mathcal{  G}^{m-1} \phi_{\overline K}(x) \bigr\rVert_{L_{ \mathrm{P}}^2} \\ 
& \leq \bigl\lVert \mathcal{  S} \bigl( \widehat{  \mathcal{  Q}}^{m-1} - \mathcal{  G}^{m-1} \bigr) \phi_{\widehat K}(x) \bigr\rVert_{L_{ \mathrm{P}}^2} 
+ \bigl\lVert \mathcal{  S} \mathcal{  G}^{m-1} \bigl( \phi_{\widehat K}(x) - \phi_{\overline K}(x)  \bigr) \bigr\rVert_{L_{ \mathrm{P}}^2}.  
\end{align*}
Since $\lVert \mathcal{  S} u \rVert^2_{L_{ \mathrm{P}}^2} = \langle \mathcal{  S}^* \mathcal{  S} u, u \rangle_{\mathcal{  H}}= \langle \mathcal{  G} u, u \rangle_{\mathcal{  H}}$, then
$\lVert \mathcal{  S} \rVert_{\infty} = \lVert \mathcal{  G} \rVert^{1/2}_{\infty} = 1$
and hence,
recalling the definition of $\phi_{\widehat K},$ we get
\[
\mathcal{  E}_r(x) 
\leq \lVert \widehat{  \mathcal{  Q}}^{m-1} - \mathcal{  G}^{m-1} \rVert_{\infty} 
\lVert \chi \rVert_{\infty} \lVert \phi_{\overline K}(x) \rVert_{\mathcal{  H}}
+ \lVert \chi - 1 \rVert_{\infty} \lVert \phi_{\overline K}(x) \rVert_{\mathcal{  H}}. 
\]

\vskip1mm
\noindent
We now prove the second bound. Let $\mathbf{\Pi} : L_{ \mathrm{P}}^2 \to L_{ \mathrm{P}}^2$ be 
the orthogonal projector on $\im(\mathbf{G})$. Since, according to ~\cref{R:imG}, almost surely, $\widehat{  \mathbf{R}}_x - \mathbf{R}_x \in \im(\mathbf{G})$,  for any $x \in \mathcal{  X}$, then
\[ 
\Bigl\langle \widehat{  \mathbf{R}}_x - \mathbf{R}_x, f \Bigr\rangle_{L_{ \mathrm{P}}^2} =  
\Bigl\langle \widehat{  \mathbf{R}}_x - \mathbf{R}_x, \mathbf{\Pi}(f) \Bigr\rangle_{L_{ \mathrm{P}}^2} \quad \text{ almost surely.} 
\] 
Moreover, since $\im(\mathbf{G}) = \mathcal{  S} \bigl( \im ( \mathcal{  G} ) \bigr)$, there is $u \in 
\im ( \mathcal{  G} )$ such that $ \mathbf{\Pi}(f) = \mathcal{  S}u$. We can then write
\begin{align*}
\Bigl\langle \widehat{  \mathbf{R}}_x - \mathbf{R}_x, f \Bigr\rangle_{L_{ \mathrm{P}}^2} 
& =  \Bigl\langle \widehat{  \mathbf{R}}_x - \mathbf{R}_x, \mathcal{  S} u \Bigr\rangle_{L_{ \mathrm{P}}^2}  \\ 
& = \Bigl\langle \mathcal{  S} \bigl( \widehat{  \mathcal{  Q}}^{m-1} \phi_{\widehat K}(x)  - \mathcal{  G}^{m-1} \phi_{\overline K}(x) \bigr), \mathcal{  S} u \Bigr\rangle_{L_{ \mathrm{P}}^2} \\
& = \bigl\langle \widehat{  \mathcal{  Q}}^{m-1} \phi_{\widehat K}(x) - \mathcal{  G}^{m-1} \phi_{\overline K}(x), \mathcal{  G} u \bigr\rangle_{\mathcal{  H}} \\ 
& = \bigl\langle \bigl( \widehat{  \mathcal{  Q}}^{m-1} - \mathcal{  G}^{m-1} \bigr) \phi_{\widehat K}(x), \mathcal{  G} u \bigr\rangle_{\mathcal{  H}} 
+ \bigl\langle \mathcal{  G}^{m-1} \bigl( \phi_{\widehat K}(x) - \phi_{\overline K}(x) \bigr), \mathcal{  G} u \bigr\rangle_{\mathcal{  H}}.
\end{align*}
Therefore, similarly as before, we get
\begin{align*}
\mathcal{  E}_c(f) 
& \leq \bigl\lVert { \widehat{ \mathcal G}}^{1/2} \bigl( \widehat{  \mathcal{  Q}}^{m-1} - \mathcal{  G}^{m-1} \bigr) \mathcal{  G} u \rVert_{\mathcal{  H}} + 
\lVert \chi - 1 \rVert_{\infty} \Biggl( \int \bigl\langle  \mathcal{  G}^{m-1} \phi_{\overline K}(x), \mathcal{  G} u \bigr\rangle^2_{\mathcal{  H}} \, \mathrm{d}\mathrm{P}(x) \Biggr)^{1/2} \\
& = \bigl\lVert { \widehat{ \mathcal G}}^{1/2} \bigl( \widehat{  \mathcal{  Q}}^{m-1} - \mathcal{  G}^{m-1} \bigr) \mathcal{  G} u \rVert_{\mathcal{  H}} + \lVert \chi - 1 \rVert_{\infty} \bigl\lVert \mathcal{  G}^{m+1/2} u \bigr\rVert_{\mathcal{  H}}\\ 
& \leq \lVert \chi \rVert_{\infty} \lVert \widehat{  \mathcal{  Q}}^{m-1} - \mathcal{  G}^{m-1} \rVert_{\infty} 
\lVert \mathcal{  G}^{1/2} u \rVert_{\mathcal{  H}} + \lVert \chi - 1 \rVert_{\infty} \lVert \mathcal{  G}^{1/2} u \rVert_{\mathcal{  H}}.
\end{align*}
We conclude observing that 
\[
\lVert \mathcal{  G}^{1/2} u \rVert_{\mathcal{  H}} =\langle \mathcal{  S} u, \mathcal{  S} u \rangle_{L_{ \mathrm{P}}^2}  = \lVert \mathbf{\Pi}(f) \rVert_{L_{ \mathrm{P}}^2} \leq \lVert f \rVert_{L_{ \mathrm{P}}^2}.
\]

\subsection{Proof of Proposition 10}\label{pf2} 

Observe that 
\[
\lVert \widehat{  \mathcal  Q} - \mathcal{  G} \rVert_{\infty} 
\leq \lVert \widehat {  \mathcal  Q} - \overline{\mathcal G}  \rVert_{\infty} +  \lVert \overline{\mathcal G}  - \mathcal{  G} \rVert_{\infty}.
\]
Moreover, for any $u \in \mathcal H$, such that $\| u\|_{\mathcal H} =1,$ recalling that $\phi_{\overline K} (x)  = \mu(x)^{-1/2} \phi_{K}(x),$
\begin{align*}
\langle  \widehat {  \mathcal  Q} u, u \rangle_{\mathcal H} - \langle  \overline{\mathcal G}  u, u \rangle_{\mathcal H}
& = \frac{1}{n} \sum_{i=1}^n \left( \widehat \mu(X_i)^{-1} -  \mu(X_i)^{-1} \right) \langle u , \phi_K (X_i) \rangle_{\mathcal H}^2\\
& = \frac{1}{n} \sum_{i=1}^n \mu(X_i)^{-1} \left(  \chi(X_i)^2 - 1\right)  \langle u , \phi_K (X_i) \rangle_{\mathcal H}^2.
\end{align*}
Thus
\begin{align*}
 \lVert \widehat {  \mathcal  Q} - \overline{\mathcal G}  \rVert_{\infty} 
 & \leq \| \chi^2 - 1\|_{\infty} \sup_{\| u\|_{\mathcal H} =1}  \frac{1}{n} \sum_{i=1}^n \mu(X_i)^{-1} \langle u , \phi_K (X_i) \rangle_{\mathcal H}^2\\
 & = \| \chi^2 - 1\|_{\infty} \sup_{\| u\|_{\mathcal H} =1}  \frac{1}{n} \sum_{i=1}^n \langle u , \phi_{\overline K} (X_i) \rangle_{\mathcal H}^2
 = \| \chi^2 - 1\|_{\infty}  \| \overline{\mathcal G} \|_{\infty}.
 \end{align*}
Using the fact that $\| \mathcal G\|_{\infty} =1$ we conclude that 
\[
\lVert \widehat {  \mathcal  Q} - \overline{\mathcal G}  \rVert_{\infty} 
\leq  \| \chi^2 - 1\|_{\infty} \left(1 + \|\overline{\mathcal G}  - \mathcal G \|_{\infty}\right),
\]
which proves the proposition.

% Acknowledgements should go at the end, before appendices and references

%\acks{The results presented in this paper 
%were obtained while the author was preparing her PhD under the 
%supervision of Olivier Catoni at the D\'epartement de Math\'ematiques et Applications, \'Ecole Normale Sup\'erieure, Paris, with the financial support of the 
%R\'egion \^Ile de France.}
%
%\vskip 0.2in

\newpage
\bibliographystyle{plain} 
\bibliography{bibliothese}

\end{document}